\documentclass{amsart}
\usepackage[english]{babel}
\usepackage{amsfonts}
\usepackage{amssymb}
\usepackage{hyperref}

\usepackage{amsmath}
\usepackage{amsthm}
\usepackage{amssymb}
\usepackage{enumerate}
\usepackage{tikz}
\usepackage{tikz-cd}
\usepackage{varwidth}
\usepackage{pigpen}
\usepackage{mathabx}
\usepackage{graphicx}
\usepackage{tensor}
\graphicspath{ {./images/} }
\usetikzlibrary{decorations.pathmorphing}
\usepackage{etoolbox}

\usetikzlibrary{shapes,backgrounds, calc, angles, positioning}

\newtheorem{theorem}{Theorem}
\newtheorem{lemma}[theorem]{Lemma}
\newtheorem{corollary}[theorem]{Corollary}

\theoremstyle{definition}
\newtheorem{notation}[theorem]{Notation}
\newtheorem{remark}[theorem]{Remark}
\newtheorem{definition}[theorem]{Definition}
\newtheorem{example}[theorem]{Example}

\numberwithin{theorem}{section}
\numberwithin{equation}{theorem}

\newcommand{\CA}{\mathcal{A}}
\newcommand{\CB}{\mathcal{B}}

\newcommand{\CE}{\mathcal{E}}

\newcommand{\CG}{\mathcal{G}}
\newcommand{\CH}{\mathcal{H}}

\newcommand{\CW}{\mathcal{W}}

\newcommand{\kk}{\Bbbk}

\newcommand{\NN}{\mathbb{N}}

\newcommand{\s}{s}
\renewcommand{\t}{t}

\newcommand{\fiber}[2]{\!\tensor*[^{}_{#1}]{\times}{^{}_{#2}}}
\renewcommand{\subset}{\subseteq}

\newcommand{\dif}[1]{\ifstrequal{#1}{}{\mathop{d\,\!}}{\mathop{d#1}}}

\DeclareMathOperator{\Id}{Id}

\newcommand{\grpd}{\rightrightarrows}

\newcommand{\inv}{^{-1}}

\newcommand{\comment}[1]{}

\newcommand{\coto}{\overset{co}{\rightarrow}}
\usepackage{mathdots}

\newcommand{\nerv}[1]{^{(#1)}}

\newcommand{\bt}{\mathbf{t}}
\newcommand{\bs}{\mathbf{s}}
\newcommand{\bu}{\mathbf{u}}
\DeclareMathOperator{\Ch}{Ch}
\DeclareMathOperator{\DG}{DG}
\DeclareMathOperator{\DC}{DC}
\DeclareMathOperator{\shift}{shift}

\newcommand{\twodarrow}{\arrow[d, shift left] \arrow[d, shift right]}
\newcommand{\threedarrow}{\arrow[d, shift left, shift left] \arrow[d, shift right, shift right] \arrow[d]}
\newcommand{\fourdarrow}{\arrow[d, shift left, shift left, shift left] \arrow[d, shift right , shift right, shift right] \arrow[d, shift left] \arrow[d, shift right]}
\title{Simplicial sheaves of modules and Morita invariance of groupoid cohomology}
\author{Xiang Tang}
\address{Department of Mathematics, Washington University, St. Louis, MO, U.S.A., 63130.}
\author{Joel Villatoro}
\address{Department of Mathematics, University of Indiana, Bloomington, IN, U.S.A., 47405.}

\begin{document}
\begin{abstract}In this article we develop a unified framework for proving Morita invariance of cohomology theories associated to Lie groupoids. Our approach is to view these cohomology theories as arising from sheaves of modules on the nerve of the groupoid. We establish criteria for when such sheaves of modules give rise to Morita invariant cohomology theories.
\end{abstract}

\maketitle

\section{Introduction}
Briefly, a groupoid is a small category with invertible morphisms. More concretely, a groupoid consists of two sets \( \CG_0 \) and \( \CG_1 \) called the set of \emph{objects} and the set of \emph{arrows}, respectively. These two sets are equipped with the following additional structure:
\begin{itemize}
\item The source $s^\mathcal{G}$ and target $t^\mathcal{G}$ are functions from $\mathcal{G}_1$ to $\mathcal{G}_0$ mapping a morphism $g\in \mathcal{G}$ to elements in $\mathcal{G}_0$. From a category theoretic point of view, the source and target functions represent the domain and codomain, respectively.
\item The multiplication $m$ is a function from $\mathcal{G}^{(2)}=\mathcal{G}_1 {_s\times _t}\mathcal{G}$ to $\mathcal{G}_1$ mapping a pair of composable morphisms, i.e. $(g_1, g_2)\in \mathcal{G}_1\times \mathcal{G}_2$ such that $s(g_1)=t(g_2)$, to the composition $g_1g_2\in \mathcal{G}_1$. 
\item The identity morphism, which is a function $i: \mathcal{G}_0\to \mathcal{G}_1$.
\item The inverse of a morphism gives a bijection from $\mathcal{G}_1$ to itself. 
\end{itemize}
Note that if the set of objects \( \CG_0 \) is a singleton, then this data is equivalent to the data needed to define a group.

Similar to Lie groups, Lie groupoids are groupoids $\mathcal{G}$ where both $\mathcal{G}_1$ and $\mathcal{G}_1$ are smooth manifolds\footnote{$\mathcal{G}_1$ might not be Hausdorff.} with smooth structure maps\footnote{We often require that both $s^{\mathcal{G}}$ and $t^{\mathcal{G}}$ are submersions.}. Lie groups are useful as very general models for studying the symmetries of a space. From this point of view, one thinks of a Lie groupoid \( \CG \) as encoding symmetry data over the collection of objects \( \CG_0 \). For example, if a Lie group \( G \) acts on a smooth manifold \( M \) then we can form a Lie groupoid with arrows \( \CG_1 := G \times M \) and objects \( \CG_0 := M \).

Since Lie groupoids represent very general models for symmetry, they are also useful as models for the study quotient spaces where we take the quotient of a manifold \( M \) via some prescribed symmetry. Indeed, given a Lie groupoid $\mathcal{G}$, we can say two elements $x,y\in \mathcal{G}_0$ are equivalent, denoted by $x\sim y$, if there is a morphism $g\in \mathcal{G}_1$ such that $x=s(g)$ and $y=t(g)$. This is to be thought of as the equivalence relation induced on \( \CG_0 \) arising from symmetries encoded by \( \CG \). For example, in the case of a \( G \) action on a manifold \( M \) the equivalence classes on \( \CG_0 \) are precisely the orbits of the \( \CG \) action.

In general, we call the equivalence classes of \( \CG_0 \) the \emph{orbits} of the Lie groupoid \( \CG \) and \( \CG_0/\sim \) is called the \emph{orbit space}. Now, in general, the orbit space \( \CG_0/\sim \) is not a manifold and can have very coarse topology (few continuous functions). The orbit space \( \CG_0 /\sim \) is sometimes called the \emph{coarse moduli space} and is often a rather unsatisfying model for the singular space. An alternative to this purely topological approach to the quotient arises from viewing Lie groupoids as \emph{differentiable stacks}. The idea behind the differentiable stack approach is that one should treat the entire Lie groupoid as a sort of atlas for the quotient space. The notion of Morita equivalence of Lie groupoids is defined\footnote{There are several equivalent characterizations of Morita equivalence (see \cite{Behrend-Xu}). One version says two groupoids are Morita equivalent if they have equivalent categories of principle bundles. Another version says Lie groupoids are Morita equivalent if there exists a ``smooth'' equivalence of categories between them (see Sec. \ref{subset:weak equivalence}).} in a way that, philosophically speaking, means two Lie groupoids are Morita equivalent when they are both atlases for ``isomorphic'' differentiable stacks. Towards this end, it is interesting to study properties of Lie groupoids that are preserved by Morita equivalence (i.e. Morita invariants) as they are thought of as natural invariants associated to the underlying singular space.

This article is concerned with studying Morita invariants that arise from various natural cochain complexes that one can associate to Lie groupoids. The most fundamental of these cochain complexes is simply called the \emph{groupoid cohomology} which is formed from the set of \emph{groupoid cochains} (see Definition~\ref{defn:simplicial.cohomology}). This complex (and hence its cohomology) is thought of as a model for the algebra of functions on the orbit space. In fact, when a Lie groupoid is Morita equivalent to a smooth manifold (a space equipped with only trivial symmetries), this complex is quasi-isomorphic to the vector space of functions on the quotient space.

 Moving beyond groupoid cochains, there are other natural complexes and cohomologies that come from closely related structures. For example, one can extend groupoid cochains to groupoid cochains with values in a representation of a Lie groupoid on a vector bundle or some weak version thereof. This complex provides us with a model for the space of sections of a vector bundle over the quotient space. Another example is the Bott-Shulmann cohomology of a Lie groupoid, which provides a model for the deRham cohomology of the underlying singular space. As natural generalizations of group cohomology, groupoid cohomologies are well studied in literature with applications in the study of gerbes, $C^*$-algebras, dynamical systems, Lie theory, etc., c.f. \cite{Behrend-Xu, Haefliger, Muhly-Williams, Renault}. In particular, there are many articles in the literature concerned with proving Morita invariance of these and other closely related cohomologies, c.f. \cite{AriasAbad, Crainic-Moerdijk, del-hoyo, Li-Bland-Meinrenken, Meinrenken-Salazar, Xu-Weinstein, Tu_Grpd_Coh}.

The aim of this article is to provide a unified framework for proving that these kinds of cohomology theories are Morita invariants. Rather than having many different specialized approaches adapted to each context, in this article, we will show that one can establish the Morita invariance of these different cohomology theories by viewing them as different examples of one structure: a well-behaved sheaf of complexes over the nerve of a groupoid. We will give a set of conditions which are sufficient to conclude that the cohomology of such a sheaf of complexes indeed constitutes a Morita invariant.

\subsection*{Main Results}
In order to state our main results, let us establish some terminology and observations:
\begin{itemize}
    \item A smooth homomorphism of Lie groupoids \( F \colon \CH \to \CG \) is called a \emph{weak equivalence} if its an equivalence of categories and \( F_0 \colon \CG_0 \to \CG_0 \) is transverse to the orbits of \( \CG_0 \). Two Lie groupoids are Morita equivalent if and only if they can be connected by a zig-zag of weak equivalences\footnote{For the purposes of this article this is essentially the \emph{definition} of Morita equivalence.}.
    \item A \emph{cosimplicial complex} on a Lie groupoid \( \CG \) consists of a sheaf of complexes defined on the nerve of \( \CG \) which has a (co)simplicial structure compatible with the existing simplicial structure on \( \CG \). Such cosimplicial complexes are called \emph{good} if they satisfy a sort of Kan-like condition (see Definitions~\ref{definition:cosimplicial.module} and Definition~\ref{definition:cosimplicial.complex}). Every cosimplicial complex has a natural cohomology associated to it that we denote with \( H(\CE) \).
    \item A \emph{sheaf of complexes on the big site} \( \CE \) is a sheaf on the site of smooth manifolds which assigns to each manifold \( M \) a complex \( \CE(M) \) of \( C^\infty_M\)-modules (think, for example, the sheaf of differential forms \( \Omega \) of all degrees).
    \item Given a homomorphism of Lie groupoids \( F \colon \CH \to \CG \) and a cosimplicial complex \( \CE \) on \( \CG \) then it is possible to construct a ``pull-back'' cosimplicial complex \( F^* \CE  \) on \( \CH \).
    
\end{itemize}
We can now state our main theorems:
\begin{theorem}[From Theorem \ref{theorem:morita.invariance.for.pullback.complexes} in body]
    Suppose \( F \colon \CH \to \CG \) is a weak equivalence of Lie groupoids and \( \CE \) is sheaf of complexes on the big site. Then \( H(\CE(\CG) )  \cong H(\CE(\CH)) \).
\end{theorem}
\begin{theorem}[From Theorem \ref{thm:big.site.for.complexes} in body]
    Suppose \( F \colon \CH \to \CG \) is a weak equivalence of Lie groupoids and \( \CE \) is a good cosimplicial complex on \( \CG \). Let \( F^* \CE \) denote the pullback cosimplicial complex on \( \CG \). Then:
    \[ H(\CE) \cong H(F^* \CE) \]
\end{theorem}
Several well-known results regarding Morita invariant cohomologies follow from these theorems as corollaries. Our first two corollaries relate the complexes of groupoid cochains of Morita equivalent groupoids.
\begin{corollary}[\cite{Cra03}\footnote{This article is the earliest appearance of this fact in the literature as far as the authors are aware.}]
    Suppose \( \CG \) is a Lie groupoid and \( \Ch(\CG)  \) denote the complex of groupoid cochains on \( \CG \). If \(\CG  \) and \( \CH \) are Morita equivalent Lie groupoids then we have that:
    \[ H(\Ch(\CG)) \cong H(\Ch(\CH)) \]
\end{corollary}
\begin{corollary}[Originally proved in \cite{AriasAbad}]
    Suppose \( \CG \) is a Lie groupoid and suppose we have a representation of \( \CG \) on a vector bundle \( E \to \CG_0 \). Let \( \Ch(\CG,E)  \) be the associated cochain complex of \( \CG \)-cochains with coefficients in \( E \). If \( F \colon \CH \to \CG\) is a weak equivalence and \( \Ch(\CH,F^* E )  \) is the cochain complex of \( F^*E \)-valued cochains then we have that:
    \[ H(\Ch(\CG,E)) \cong H(\Ch(\CH,F^* E))  \]
\end{corollary}
Another corollary relates the cohomologies of the Bott-Schulman double complex of Morita equivalent Lie groupoids.
\begin{corollary}
    Suppose \( \CG \) is a Lie groupoid and let \( \Omega^\bullet(\CG\nerv{\bullet} ) \) be the Bott-Schulman double complex. Write \( H_{BS}(\CG) \) to denote the total cohomology of this double complex. If \( \CG \) and \( \CH \) are Morita equivalent groupoids then:
    \[ H_{BS} (\CG) \cong H_{BS} (\CH) \]
\end{corollary}
Our last corollary is more of a corollary of the proofs than the actual conclusions of Theorem 1.1 and 1.2. It says that the conclusions of these theorems still hold if we pass to the appropriate notion of compactly supported cohomology.
\begin{corollary}
    Given \( \CE \) an arbitrary cosimplicial complex over \( \CG \). Let \( \CE_C\) denote the (pre)sheaf of sections of \( \CE \) which are compactly supported (relative to the topology on the orbit space \( \CG_0/\CG_1 \). Then \( \CE_C \) is a subcomplex with associated cohomology denoted \( H_C(\CE) \).
    
    Theorem 1.1 still holds if we replace the last equation with:
    \[ H_C(\CE) \cong H_C(F^* \CE ) \]
    and Theorem 1.2 still holds if we replace the last equation with:
    \[ H_C(\CE(\CG)) \cong H_C(\CE(\CH)) \]
\end{corollary}

\subsection*{Our Approach}
Our philosophy is to approach the topic in a relatively low-tech way. In particular, we wish to write the article in a way accessible to those mainly familiar with the smooth theory of Lie groupoids without relying on higher structures more than necessary. The primary tools we will use come from the structure and theory of simplicial manifolds and cosimplicial vector spaces. 

Much of what we discuss in this article can be generalized beyond the context of Lie groupoids and can be adapted to the study of simplicial manifolds (satisfying some appropriate Kan condition). However, fully extending our work would require a bit of theoretical development beyond what we present here. This could be the topic of a future article if there is interest.

A key component of our approach is the notion of a Decalage (or the shift double of a simplicial set). The shift double of a simplicial set is a way of converting any simplicial set into a homotopy equivalent (and relatively simple) double simplicial set. The proofs of our main theorems rely on showing that the associated morphism at the level of complexes factors through the shift double in a ``nice'' way.

The article is organized as follows. A brief introduction to the background and notations is included in Section \ref{sec:background}. We explain the Decalages construction in Section \ref{sec:shifting simplicial structures}. In Section \ref{sec:augumentations and retracts}, we discuss retracts of simplicial manifolds and cosimpicial vector spaces. The discussion of shifted retractions of groupoids is presented in Section \ref{sec:shifted retractions of groupoids}. In Section \ref{sec:groupoid homomorphisms}, we explain the double groupoid associated with a groupoid homomorphism. Cohomology of sheaves of cosimplicial modules on a groupoid $\mathcal{G}$ is introduced in Section \ref{sec:cosimplicial modules over groupoids}. Morita invariance of cohomology of sheaves of cosimplicial modules on groupoids, Theorem \ref{theorem:morita.groupoid.cohomology}, is established in Section \ref{sec:morita invariance}. We prove a generalized version of Morita invariance of the Bott-Shulman cohomology via sheaves of cosimplicial complexes on groupoids (Theorem \ref{thm:big.site.for.complexes}), in Section \ref{sec:cosimplicial complexes over groupoids}.\\

\noindent{\bf Acknowledgments}: We would like to thank 
Marius Crainic, Matias del Hoyo, and Rajan Mehta for inspiring discussions. Villatoro's research was partially supported by the NSF grant DMS-2137999. Tang's research was partially supported by the NSF grants DMS-1952551, DMS-2350181, and Simons Foundation grant MPS-TSM-00007714.

\section{Background and Notation}
\label{sec:background}
\subsection{Lie groupoids notation}
If \( \CG \) is a Lie groupoid we write \( \CG_1 \) and \( \CG_0 \) to denote the arrows and objects respectively. We use the notation:
   \[ \CG\nerv{n} := \overbrace{\CG_1 \fiber{\bs}{\bt} \CG_1 \fiber{\bs}{\bt} \cdots \fiber{\bs}{\bt} \CG_1}^{n\text{-times}} \]
to denote the space of \( n \) composable arrows.

Many of the objects we consider will have their own notions of ``source'' and ``target.'' If necessary to disambiguate we may use superscripts \( \bs^\CG \colon \CG_1 \to \CG_0 \) or \( \bt^\CG \colon \CG_1 \to \CG_0 \) to indicate we are specifically referring to the source map (or target map) for the groupoid \( \CG \). However, in order to reduce clutter in our notation we may often omit these superscripts when it is clear from context. We will also consider the source and target maps for a groupoid to be defined for tuples of composable arrows with the convention:
\[ \bs(g_n,\ldots , g_1) := \bs(g_1) \qquad \bt(g_n, \ldots, g_1) := \bt(g_n) \]
Hence, for example, given a smooth function \( f \colon N \to \CG_0 \) we have that:
\[ \CG\nerv{n} \fiber{\bs}{f} N = \{ (g_n, \ldots , g_1, x) \in \CG^n \times N \ : \ (g_n,\ldots,g_1) \in \CG\nerv{n} \text{ and } \bs(g_1) = f(x) \}  \]

\subsection{Simplicial and cosimplicial structures}
The sets of composable arrows of a Lie groupoid have a special structure which is formally referred to as a simplicial set (or manifold in the case of a Lie groupoid).
\begin{definition}\label{defn:simplicial.set.concrete}
    A \emph{simplicial set} \( S \) consists of the following data: 
    \begin{itemize}
        \item a collection of sets indexed by natural numbers \( \{ S\nerv{n} \}_{n \in \NN}\);
        \item for each \( n \ge 0 \) and \( 0 \le i \le n \), a function:
    \[d^n_i \colon S\nerv{n} \to S\nerv{n-1} \]
    called the \emph{face maps}.
    \item for each \( n \ge 0 \) and \( 0 \le i \le n\), a function:
    \[  s^n_i \colon S\nerv{n} \to S\nerv{n+1}  \]
    called \emph{degeneracy maps}.
    \end{itemize}
    The face maps and degeneracy maps are required to satisfy the following relations:
    \begin{enumerate}[(SS1)]
        \item \( \forall 0 \le i < j \le n+1\) 
        \[ d^n_i \circ d^{n+1}_j = d^n_{j-1} \circ d^n_i  ; \]
        \item \( \forall 0 \le i < j \le n+1 \), \( n \ge 1\) 
        \[ d^{n+1}_i \circ s^{n}_j = s^{n-1}_{j-1} \circ d^n_i;
        \]
        \item \( \forall 0 \le i < n \) 
        \[ d^{(n+1)}_i \circ s^{n}_i = d^{{n+1}}_{i+1} \circ s^{n}_{i} = \Id ; 
        \]
        \item \( \forall 0 \le j +1 < i \le n+1 \), \( n \ge 1\)
        \[ d^{n+1}_i \circ s^{n}_{j} = s^{n-1}_j \circ d^n_{i-1};
        \]
        \item \( \forall 0 \le i <  j \le n+1 \),
        \[ s^{n+1}_j \circ s^n_i = s^{n+1}_{i} \circ s^n_{j-1}. 
        \]
    \end{enumerate}
\end{definition}
Although there are many kinds of simplicial sets. The most important kind for our purposes arises from a groupoid. The simplicial set associated to a groupoid is called the nerve of the groupoid.

Since the structure of a groupoid is completely determined by its nerve, we may sometimes conflate the two objects.
\begin{definition}
    Suppose \( \CG_1 \grpd \CG_0 \) is a groupoid. The \emph{nerve} of \( \CG \) is the simplicial set \(\{ \CG\nerv{n} \}_{n \in \mathbb{N}} \) where we recall that:
    \[ \CG\nerv{n} := \overbrace{\CG_1 \fiber{\s}{\t} \CG_1 \fiber{\s}{\t} \cdots \fiber{\s}{\t} \CG_1}^{n\text{-times}}. \]
    By convention, we let \( \CG\nerv{0} = \CG_0 \) be the space of objects.
    The face maps are defined to be:
        \[ d^n_i(g_n,\ldots, g_1) := \begin{cases}
        (g_n,\ldots,g_2), & i = 0, \\
        (g_n,\ldots, g_{i+2},\  g_{i+1} \cdot g_i\ , g_{i-1}, \ldots g_1), & 0< i < n,\\
        (g_{n-1}, \ldots , g_1), & i = n,
    \end{cases}\]
    when \( n \ge 2\). In the case \( n = 1\) the face maps are just the target and source maps:
    \[ d^1_0 = \bt \qquad d^1_1 = \bs . \]
    The degeneracy maps are:
    \[  s^n_i(g_{n}, \ldots , g_1) := (g_{n}, \ldots , g_{i+1}, (u\circ t)(g_i) , g_i , \ldots , g_1 )\]
    where, in the above \( u \colon M \to \CG \) is the unit embedding. By convention \( s^0_0 := u \). 
    We leave it to the reader to verify that the axioms of simplicial sets hold.
\end{definition}
In general, a simplicial object in a category consists of a collection of objects in that category, together with morphisms that satisfy axioms (SS1-5). However, in some cases it is reasonable to require some additional surjectivity assumptions. This is the case for simplicial manifolds:
\begin{definition}\label{defn:simplicial.manifold}
    A \emph{simplicial manifold} is a simplicial set \( M = \{ M\nerv{n} \}_{n \in \NN } \) where each \( M\nerv{n}\) is equipped with the structure of a smooth manifold and the face and degeneracy maps are smooth maps of manifolds. Finally, we require that the face maps are submersions.
\end{definition}
The nerve of a Lie groupoid is an example of a simplicial manifold.

In the same vein, there is a concrete definition of a cosimplicial set in terms of sets and functions. We will write out the definition here for completeness but, as we see, the only distinction is reversing domain and codomain, as well as the order of composition.
\begin{definition}\label{defn:cosimplicial.set.concrete}
A \emph{cosimplicial set} \( S \) consists of the following data:
   \begin{itemize}
        \item a collection of sets indexed by non-negative natural numbers \( \{ S\nerv{n} \}_{n \in \NN}\);
        \item for each \( n \ge 0 \) and \( 0 \le i \le n \), functions:
    \[\phi^n_i \colon S\nerv{n-1} \to S\nerv{n} \]
    called the \emph{coface maps};
    \item for each \( n \ge 0 \) and \( 0 \le i \le n\), functions:
    \[  \sigma^n_i \colon S\nerv{n+1} \to S\nerv{n}  \]
    called \emph{codegeneracy maps}.
    \end{itemize}
    This data is required to satisfy the following axioms:
    \begin{enumerate}[(CS1)]
        \item \( \forall 0 \le i < j \le n+1\) 
        \[  \phi^{n+1}_j \circ \phi^n_i= \phi^{n+1}_i \circ  \phi^n_{j-1}, \]
        \item \( \forall 0 \le i < j \le n \) , \( n \ge 1\)
        \[  \sigma^{n}_j \circ \phi^{n+1}_i=   \phi^n_i \circ \sigma^{n-1}_{j-1} ,\]
        \item \( \forall 0 \le i < n \) 
        \[  \sigma^{n}_i \circ \phi^{n+1}_i =  \sigma^{n}_{i} \circ  \phi^{n+1}_{i+1} = \Id , \]
        \item \( \forall 0 \le j + 1 < i \le n+1 \), \( n \ge 1\)
        \[ \sigma^{n}_{j} \circ \phi^{n+1}_i  =  \phi^n_{i-1} \circ  \sigma^{n-1}_j,  \]
        \item \( \forall 0 \le i <  j \le n+1 \),
        \[ \sigma^{n+1}_j \circ \sigma^n_i = \sigma^{n+1}_{i} \circ \sigma^n_{j-1}. \]
    \end{enumerate}
\end{definition}

\subsection{(Co)simplicial vector spaces}
A particularly type of cosimplicial set relevant to our discussions are ones which are formed from vector spaces and linear maps. It is well known that cosimplicial vector spaces have a natural cohomology theory associated to them.

\begin{definition}\label{defn:cosimplicial.vector.space}
Let \( \kk \) be a fixed field.
    A \emph{(co)simplicial vector space} is a cosimplicial set \( V = \{ V\nerv{n} \}\) where each \( V\nerv{n} \) is a vector space and the (co)face and (co)degeneracy maps are linear maps.
\end{definition}
Throughout this article we will take the base field to be \( \mathbb{R}\) since that is the one that will be most relevant to us. However, cosimplicial vector spaces make sense over any field.

One important feature of (co)simplicial vector spaces is their simplicial (co)homology. This arises from what we will call the \emph{simplicial differential}. We will define the ``co'' version of this object but the following definitions and proofs are essentially identical for the standard case.
\begin{definition}\label{defn:cosimplicial.differential}
    Suppose \( V \) is a cosimplicial vector space. The \emph{cochain complex} of \( V \) is the graded vector space:
    \[ \Ch(V) := \bigoplus_{ n \in \mathbb{N}} V\nerv{n} .  \]
    The \emph{simplicial differential} on \( V \) is the degree one graded endomorphism 
    \[ \delta \colon \Ch(V) \to \Ch(V) \]
    with \(n\)-th homogeneous component given by:
    \[ \delta^n := \sum_{i=0}^{n+1} (-1)^i \phi^n_i \colon V\nerv{n} \to V\nerv{n+1}.
    \]
\end{definition}
It is a classical fact that this differential squares to zero as a direct consequence of the face map relations. We include the calculation below for completeness.
\begin{lemma}\label{lemma:cosimplicial.differential.squares.to.zero}
    If \( V \) is a cosimplicial vector space and \( \delta \) is the simplicial differential, then \( \delta^2 = 0\). 
\end{lemma}
\begin{proof}
We must show that
    \[ \forall n \ge 0 ,\qquad \delta^{n+1} \delta^n = 0 . 
    \]
From the definition of a cosimplicial set, we know that:
\begin{equation}\label{eqn:simplicial.differential.squares.to.zero}
    \forall 0 \le i < j \le n+1, \qquad \phi^n_j  \phi^n_i = \phi^{n+1}_i \phi^n_{j-1}. 
\end{equation}
    Using the definition of the simplicial differential observe that:
    \[ \delta^{n+1} \delta^n = \sum_{j = 0}^{n+2} \sum_{i=1}^{n+1} (-1)^{i+j} \phi^{n+2}_j \phi^{n+1}_i  \]
    This sum can be split into two parts:
    \[  \delta^{n+1} \delta^n = \left( \sum_{1 \le j \le i \le n+1} (-1)^{i+j} \phi^{n+2}_j \phi^{n+1}_i \right) + \left( \sum_{1 \le i < j \le n+2} (-1)^{i+j} \phi^{n+2}_j \phi^{n+1}_i \right)  \]
    Now if we apply Equation~\ref{eqn:simplicial.differential.squares.to.zero} to the second component of this sum we get:
   \[  \delta^{n+1} \delta^n = \left( \sum_{1 \le j \le i \le n+1} (-1)^{i+j} \phi^{n+2}_j \phi^{n+1}_i \right) + \left( \sum_{1 \le i < j \le n+2} (-1)^{i+j} \phi^{n+2}_i \phi^{n+1}_{j-1} \right)  \]
   We can re-index the second sum to obtain:
    \[  \delta^{n+1} \delta^n = \left( \sum_{1 \le j \le i \le n+1} (-1)^{i+j} \phi^{n+2}_j \phi^{n+1}_i \right) + \left( \sum_{1 \le j \le i \le n+1} (-1)^{i+j+1} \phi^{n+2}_j \phi^{n+1}_{i} \right)  \] 
    Since these two sums differ only by a sign, we conclude that \( \delta^{n+1} \delta^n = 0\).
\end{proof}
\begin{definition}\label{defn:simplicial.cohomology}
    Let \( V \) be a cosimplicial vector space and let \( \delta \) be the associated simplicial differential. The \emph{simplicial cohomology}, denoted \( H^\bullet(V) \) of \( V \) is defined to be the cohomology of the simplicial differential \( H^\bullet(\Ch(V), \delta) \).
\end{definition}
\subsection{Augmentation}

A key notion that we will make use of is that of an augmentation. Traditionally, augmentations are ways of modifying bounded complexes by adding an additional term. They arise in the study of various kinds of resolutions in homological algebra. However, this notion has an analogue for simplicial structures and cosimplicial structures.

We will begin with defining augmentation in the context of simplicial manifolds.

\begin{definition}
    Suppose \( X \) is a simplicial manifold. An \emph{augmentation} of \( X \) consists of a smooth manifold \( Y \) together with a smooth map \( p \colon X\nerv{0} \to Y \) satisfying \( p \circ d_1 = p \circ d_0 \).
\end{definition}
Augmentations of simplicial sets are sometimes thought of as extending the simplicial set by adding a set of ``\( -1 \) simplices.''
\[ 
\begin{tikzcd}
    Y & X\nerv{0} \arrow[l] & X\nerv{1} \arrow[l, shift left] \arrow[l, shift right] &  \cdots \arrow[l, shift left, shift left] \arrow[l] \arrow[l, shift right , shift right] 
\end{tikzcd}\]
Another perspective is to view an augmentation as giving a morphism of simplicial manifolds from \(X  \) to \( Y \) where \( Y \) is thought of as a simplicial manifold where all face/degeneracy maps are the identity:
\[ 
\begin{tikzcd}
    X\nerv{0} \arrow[d, "p"] & X\nerv{1} \arrow[d] \arrow[l, shift left] \arrow[l, shift right] &  \cdots \arrow[l, shift left, shift left] \arrow[l] \arrow[l, shift right , shift right] \\
    Y  & Y \arrow[l, shift left] \arrow[l, shift right] &  \cdots \arrow[l, shift left, shift left] \arrow[l] \arrow[l, shift right , shift right] 
\end{tikzcd}
\]
\begin{example}
    If \( X \) is the nerve of a Lie groupoid \( \CG \) then one can think of an augmentation as a groupoid homomorphism \( P \colon \CG \to 1_Y \) where \( 1_Y \) is the identity groupoid over \( Y \).
\end{example}
There is a corresponding definition for cosimplicial vector spaces.
\begin{definition}
    Suppose \( V \) is a cosimplicial vector space. An \emph{augmentation} of \( V \) consists of a vector space \( W \) together with a linear map \( L \colon W \to V\nerv{0} \) with the property that \( \phi_1 L = \phi_2 L\)
\end{definition}

Just like with simplicial structures, augmentations of cosimplicial structures can be thought of as extensions of the simplicial structure of \( V \) or as a map from a standard vector space into a simplicial one.
\begin{example}
If \( X \) is a simplicial manifold then \( C^\infty(X) \) inherits the structure of a cosimplicial vector space. Furthermore, if \( f \colon X \to Y  \) is an augmentation of \( X \) then the corresponding linear map \( f^* \colon C^\infty(Y) \to C^\infty(X) \) is an augmentation of \( C^\infty(X) \).
\end{example}
Recall that one can construct a natural cochain complex out of any cosimplicial vector space. The notion of an augmentation carries over to this context as well.
\begin{definition}
    Suppose \( C \) is a cochain complex. An \emph{augmentation} of \( C \) consists of a linear map:
    \[ L \colon V \to C^0\]
    with the property that \( d^0 L = 0 \).

    If \( L \) is an augmentation of \( C \) then the \emph{augmented complex} \( C_L \) is the one obtained by extending \( C \) with a ``-1'' term.
    \[
    C_L := \left( \begin{tikzcd}
        V \arrow[r, "L"] & C^0 \arrow[r] & C^1 \arrow[r] & C^2 \arrow[r] & \cdots 
    \end{tikzcd} \right)
    \]
    An augmentation \( L \) is called \emph{acyclic} if the augmented complex has trivial cohomology groups in all degrees.
\end{definition}
The term ``acyclic'' here is not necessarily standard in the existing literature. We do not use the word ``exact'' because an exact augmentation usually refers to a weaker condition\footnote{Our understanding is that an \emph{exact augmentation} is one where \( H^0(C_L) = 0 \). In other words, it is exact at degree zero.}

An equivalent characterization of an augmentation is to define it as being morphism of cochain complexes of the form:
\[
\begin{tikzcd}
    V \arrow[r] \arrow[d, "L"] & 0 \arrow[d] \arrow[r] & 0 \arrow[d] \arrow[r] & \cdots \\
    C^0 \arrow[r] & C^1 \arrow[r] & C^2 \arrow[r] & \cdots 
\end{tikzcd}
\]
Acyclic augmentations are relevant due to the following observation:
\begin{lemma}
    If \( L \) is an acyclic augmentation of \( C \) then \( H^i(C) = 0 \) for all \( i \ge 1 \) and \( L \colon V \to H^0(C) \) is an isomorphism.
\end{lemma}
If we think of \( L \) as a morphism of complexes then the above lemma amounts to the observation: ``acyclic augmentations are quasi-isomorphisms.'' In the context of complexes this fact seems rather obvious. However, there is corresponding version of this fact for double and triple complexes which are less obvious and rather useful.

Augmentations of simplicial and cosimplicial structures directly correspond to augmentations of the associated cochain complexes.
\begin{example}
    Suppose \( V \) is a cosimplicial vector space. Let \( \Ch(V) \) be the associated cochain complex. Then \( L \colon W \to V \) is an augmentation of \( V \) as a cosimplicial vector space if and only if \( L \) is an augmentation of \( \Ch(V) \) as a cochain complex.
\end{example}

\subsection{Groupoid cohomology}
We saw that, associated to any cosimplicial vector space, there is an associated cochain complex and cohomology. It is worth paying particular attention to a special case where the cosimplicial vector space arises from a Lie groupoid.

The motivation behind studying this cohomology theory is that it gives a sort of model for the space of functions on the singular space (i.e. the stack) associated to \( \CG \). In fact, when a Lie groupoid is Morita equivalent to a smooth manifold, it turns out that the groupoid cohomology is isomorphic to the space of functions on the orbit space.
\begin{definition}
    Suppose \( \CG \) is a Lie groupoid. For a natural number \( n \) the \emph{cosimplicial algebra} of \( \CG \), written \( \Ch(\CG) \), is defined so that: 
    \[\Ch(\CG)\nerv{n} := C^\infty(\CG\nerv{n}). \] 
    The coface and codegeneracy maps are the pullbacks along the face and degeneracy maps of the nerve:
    \[ \phi^n_i := (d^n_i)^* \qquad \sigma^n_i := (s^n_i)^*. \]
\end{definition}
\begin{definition}
    The \emph{groupoid cochain complex}, also denoted \( \Ch(\CG) \):
    \[ \Ch(\CG) := \left(  
    \begin{tikzcd}
        \CA\nerv{0} \arrow[r, "\delta"] &  \CA\nerv{1} \arrow[r, "\delta"] & A\nerv{2} \arrow[r, "\delta"] & \cdots 
    \end{tikzcd} \right) .
    \]
    The differential here is just the simplicial differential for \(C^\infty(\CG) \). In this case it takes the form of an alternating sum of pullbacks along face maps:
    \[ \delta^n := \sum_{i=0}^{n+1} (-1)^i (d^{n+1}_i)^* \colon C^\infty(\CG\nerv{n}) \to C^\infty(\CG\nerv{n+1} ) .\]
    The cohomology:
    \[ H^{\bullet} (\CG) := H^\bullet(\CA) \]
    is called the \emph{groupoid cohomology} of \( \CG \).
\end{definition}

\subsection{Double complexes}
Double complexes will be a key component of our technical arguments. However, we will not need a full treatment of the topic. In particular, we will not make any (explicit) references to spectral sequences. The reason for this is that the actual properties of double complexes that we rely on will be fairly elementary.

\begin{definition}
    \emph{A double (cochain) complex} consists of a bigraded vector space \( D = \bigoplus\limits_{p,q \in \mathbb{N}} D^{p,q} \) equipped with two bigraded linear endomorphisms:
    \[  d_{H}: D^{\bullet,\bullet} \to D^{\bullet,\bullet+1} \qquad d_V \colon D^{\bullet,\bullet} \to D^{\bullet + 1, \bullet}  \]
    of degree \( (0,1) \) and \( (1,0) \), respectively, satisfying
    \[ d_H^2 = 0, \qquad d_V^2 = 0, \qquad d_H d_V = d_V d_H . \]
    Such a double complex can be illustrated by the following diagram:
\[
\begin{tikzcd}[column sep=large, row sep=large]
 \vdots &  \vdots &  \vdots &  \iddots \\
D^{2,0}  \arrow[r , "{d_H}"] \arrow[u, "{d_V}"] & D^{2,1}  \arrow[r , "{d_H}"] \arrow[u, "{d_V}"] & D^{2,2}  \arrow[r , "{d_H}"] \arrow[u, "{d_V}"] & \cdots \\
D^{1,0}  \arrow[r , "{d_H}"] \arrow[u, "{d_V}"] & D^{1,1}  \arrow[r , "{d_H}"] \arrow[u, "{d_V}"] & D^{1,2}  \arrow[r , "{d_H}"] \arrow[u, "{d_V}"] & \cdots \\
D^{0,0}  \arrow[r , "{d_H}"] \arrow[u, "{d_V}"] & D^{0,1}  \arrow[r , "{d_H}"] \arrow[u, "{d_V}"] & D^{0,2}  \arrow[r , "{d_H}"] \arrow[u, "{d_V}"] & \cdots \\
\end{tikzcd}
\]
\end{definition}
There are many different kinds of cohomologies that can be constructed from a double complex. The most basic ones are the horizontal and vertical cohomologies. These are bigraded vector spaces constructed by applying the cohomology functor to one direction at a time.
\begin{definition}
    \emph{The vertical cohomology} \( H_V(D) \) of a double complex \( D \) is the one obtained by taking the cohomology of each column individually. The vertical cohomology is natural bigraded and:
    \[ H_V^{ij} (D) := H^i( D^{\bullet , j} , d_V).\]
    The horizontal cohomology is defined similarly:
    \[ H_H^{ij} (D) := H^j (D^{i \bullet} , d_H ). \]
\end{definition}
Perhaps the most important cohomology of a double complex is the ``total'' cohomology.
\begin{definition}
    Given a double complex \( D^{\bullet, \bullet} \) with horizontal and vertical differentials \( d_H \) and \( d_V \), the total complex \( \mathrm{Tot}(D^{\bullet, \bullet}) \) is a defined to be:
    \[ D_{\mathrm{tot}}^n = \bigoplus_{p+q=n} D^{p,q}, \]
    with the differential \( d_{\mathrm{tot}}: D_{\mathrm{tot}}^{\bullet} \to D_{\mathrm{tot}}^{\bullet + 1} \) given by:
    \[
    d_{\mathrm{tot}} = d_H + (-1)^p d_V \quad \text{on} \quad D^{p,q}.
    \]
    The total cohomology, simply written \( H(D) \) is the cohomology of the total complex.
\end{definition}
The total cohomology is quasi-isomorphic to the diagonal cohomology, which is defined using the complex $\{D^{k,k}\}_{k\geq0}$. As we will not need this in this article, we will not discuss the diagonal cohomology further. In general, computing the total cohomology of a double complex can be very complicated. Luckily for us, the kinds of double complexes that we will see in this article are comparatively simple. Our primary tool for computing total cohomology is the notion of an augmentation.

The augmentation of a double complex is a fairly natural generalization of the notion we defined for standard complexes. However, the two-dimensional nature of double complexes means that there are actually two natural notions of augmentation.
\begin{definition}
    Suppose \( D \) is a double complex and \( V \) is a cochain complex. A \emph{left augmentation} \( L \colon V \to D \) consists of a morphism of cochain complexes:
    \[ L \colon V \to D^{\bullet,0} \]
    with the property that \( d_H L = 0 \).
    A \emph{bottom augmentation} is a morphism of cochain complexes:
    \[ B \colon W \to D^{0,\bullet} \]
    with the property that \( d_H B = 0 \).

    Left and bottom augmentations can be illustrated by a diagram like so:
\[
\begin{tikzcd}[column sep=large, row sep=large]
\vdots &   \vdots &  \vdots &  \vdots &  \iddots \\
V^2 \arrow[u,"d"] \arrow[r, dashed, "L"] & D^{2,0}  \arrow[r , "{d_H}"] \arrow[u, "{d_V}"] & D^{2,1}  \arrow[r , "{d_H}"] \arrow[u, "{d_V}"] & D^{2,2}  \arrow[r , "{d_H}"] \arrow[u, "{d_V}"] & \cdots \\
V^1 \arrow[u,"d"] \arrow[r, dashed, "L"] & D^{1,0}  \arrow[r , "{d_H}"] \arrow[u, "{d_V}"] & D^{1,1}  \arrow[r , "{d_H}"] \arrow[u, "{d_V}"] & D^{1,2}  \arrow[r , "{d_H}"] \arrow[u, "{d_V}"] & \cdots \\
V^0  \arrow[u,"d"]\arrow[r, dashed, "L"] & D^{0,0}  \arrow[r , "{d_H}"] \arrow[u, "{d_V}"] & D^{0,1}  \arrow[r , "{d_H}"] \arrow[u, "{d_V}"] & D^{0,2}  \arrow[r , "{d_H}"] \arrow[u, "{d_V}"] & \cdots \\
 & W^0 \arrow[u, dashed, "B"] \arrow[r, "d"] & W^1 \arrow[u, dashed, "B"] \arrow[r, "d"] & W^2 \arrow[u, dashed, "B"] \arrow[r, "d"] & \cdots 
\end{tikzcd}
\]
A left augmentation is called \emph{acyclic} if the augmented rows are all exact. In other words, for all \( n \in \mathbb{N} \) we have that:
\[ 
    \begin{tikzcd}
    V^n \arrow[r, "L"] & D^{n,0} \arrow[r, "d_H"] & D^{n,1} \arrow[r, "d_H"] & \cdots 
    \end{tikzcd}
\]
is exact. A bottom augmentation is called \emph{acyclic} if the augmented columns are exact.
\end{definition}
The following lemma is the main reason we care about augmentation of double complexes. It is a fairly standard fact in homological algebra by spectral sequence so we will not provide the proof here.
\begin{lemma}[Augmentation lemma]\label{lemma:augmentation.lemma}
    Suppose \( D \) is a double complex and suppose \( L \colon V \to D \) is an acyclic left augmentation of \( D \).
    
    Then \( L \) induces an isomorphism \( L_* \colon  H(V) \to H(D) \).
\end{lemma}
In other words, acyclic left/bottom augmentations are quasi-isomorphisms with respect to the total complex.

\section{Shifting (Co)simplicial structures (Decalages)}\label{sec:shifting simplicial structures}
Simplicial sets come with a lot of inherited structure. 
We will now highlight some operations one can apply to any simplicial set to construct new simplicial sets. 
In the existing literature this operation is referred to as the ``decalage'' of a simplicial set. 
In this article we will simply refer to the operation as the left/right shift of a simplicial set.

One can find more discussion of this topic in Illusie\cite{Illusie} and  but we will give a self contained explanation here. 
Our philosophy will be to keep the terminology fairly grounded and proofs quite low tech.

\subsection{Left and right shifts of simplicial structures}
\begin{definition}
    Suppose \( S \) is a simplicial set with face maps \( \{ d^n_i \} \) and degeneracy maps \( s^n_i \). If \( k \ge 0 \) is a natural number the \( k \)-th left shift of \( S \) is the simplicial set \( S_L[k]\) where:
    \[ S_L[k]\nerv{n} := S\nerv{k+n} \]
    and with face and degeneracy maps:
    \[ (d_L[k])^n_i := d^{k+n}_{i}, \qquad (s_L[k])^n_i := s^{k+n}_i . \]
    The \emph{right shift} is denoted \( S_R[k] \) and also has:
    \[ S_R[k] \nerv{n} := S\nerv{k+n} \]
    but with face and degeneracy maps:
    \[ (d_R[k])^n_i := d^{k+n}_{k+i}, \qquad (s_R[k])^n_i := s^{k+n}_{k+i} .\]
\end{definition}\begin{figure}
    \centering
\begin{tikzpicture}
    \node (f50) at (-1, 4) {$\vdots$};
    \node (f51) at (1, 4) {$\vdots$};
    \node (f52) at (3, 4) {$\vdots$};
    \node (f53) at (5, 4) {$\vdots$};
    \node (f54) at (7, 4) {$\vdots$};
    \node (f55) at (9, 4) {$\vdots$};

    \node (f40) at (0, 3) {$d^4_0$};
    \node (f41) at (2, 3) {$d^4_1$};
    \node (f42) at (4, 3) {$d^4_2$};
    \node (f43) at (6, 3) {$d^4_3$};
    \node (f44) at (8, 3) {$d^4_4$};
    
    \node (f30) at (1, 2) {$d^3_0$};
    \node (f31) at (3, 2) {$d^3_1$};
    \node (f32) at (5, 2) {$d^3_2$};
    \node (f33) at (7, 2) {$d^3_3$};
    
    \node (f20) at (2, 1) {$d^2_0$};
    \node (f21) at (4, 1) {$d^2_1$};
    \node (f22) at (6, 1) {$d^2_2$};
    
    \node (f10) at (3, 0) {$d^1_0$};
    \node (f11) at (5, 0) {$d^1_1$};

    \draw[thick] (2,4) -- (5, 1.5) -- (7, 1.5) -- (10, 4);
\end{tikzpicture}
    \caption{Face maps above the line are the ones included in the \( S_R[2] \)}
    \label{fig:shift-visualization}
\end{figure}
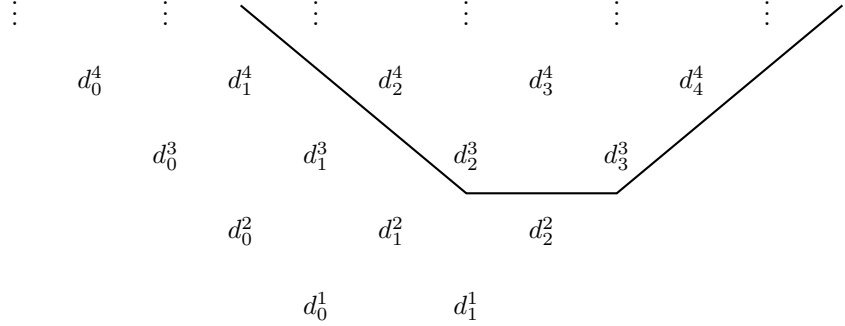
If we arrange the face maps of the simplicial set \( S \) into an infinite triangle, we can visualize the left and right shifts of \( S \) as picking out all elements above a particular node on the boundary. For example, in Figure~\ref{fig:shift-visualization} the second right shift is obtained by picking out all face maps above the face map \( d^2_2 \). Left and right shifts of cosimplicial vector spaces are defined similarly:
\begin{definition}
Suppose \( V \) is a cosimplicial vector space with coface maps \( \{ \phi^n_i \} \) and degeneracy maps \( \sigma^n_i \). If \( k \ge 0 \) is a natural number the \( k \)-th left shift of \( V \) is the simplicial set \( V_L[k]\) where:
\[ V_L[k]\nerv{n} := V\nerv{k+n} \]
and with face and degeneracy maps:
\[ (\phi_L[k])^n_i := \phi^{k+n}_{i} , \qquad (\sigma_L[k])^n_i := \sigma^{k+n}_i . \]
The \emph{right shift} is denoted \( V_R[k] \) and also has:
\[ V_R[k] \nerv{n} := V\nerv{k+n} \]
but with face and degeneracy maps:
\[ (\phi_R[k])^n_i := \phi^{k+n}_{k+i}, \qquad (\sigma_R[k])^n_i:= \sigma^{k+n}_{k+i} . \]
\end{definition}
\subsection{Shift double structures}
Through the use of left and right shifts, one obtains a double simplicial structure from any simplicial set \( S \). Since we are mainly interested in groupoids we will discuss this primarily in the context of Lie groupoids but shift double structures make sense for any simplicial object.
\begin{definition}\label{defn:shifted double}
    Let \( \CG \) be a Lie groupoid. The \emph{shift double groupoid}, denoted by \( \CG_{\shift}\) of \( \CG \) is the one obtained by arranging the left shifts of \( \CG \) into the rows and the right shifts of \( \CG \) into the columns as in the below:
\[ 
\CG_{\shift} := \left( 
\begin{tikzcd}
      \vdots \arrow[d, shift left, shift left, shift left] \arrow[d, shift right, shift right, shift right] \arrow[d, shift left] \arrow[d, shift right]  &                                               \vdots \arrow[d, shift left, shift left, shift left] \arrow[d, shift right, shift right, shift right] \arrow[d, shift left] \arrow[d, shift right]  &                                                                                    \vdots \arrow[d, shift left, shift left, shift left] \arrow[d, shift right, shift right, shift right] \arrow[d, shift left] \arrow[d, shift right] & \iddots \\
    \CG\nerv{3} \arrow[d, shift left, shift left] \arrow[d, shift right, shift right] \arrow[d] & \arrow[l, shift left] \arrow[l, shift right] \CG\nerv{4} \arrow[d, shift left, shift left] \arrow[d, shift right, shift right] \arrow[d] & \arrow[l, shift left, shift left] \arrow[l, shift right, shift right] \arrow[l] \CG\nerv{5} \arrow[d, shift left, shift left] \arrow[d, shift right, shift right] \arrow[d] & \arrow[l, shift left, shift left, shift left] \arrow[l, shift right, shift right, shift right] \arrow[l, shift left] \arrow[l, shift right] \cdots  \\
    \CG\nerv{2}  \arrow[d, shift left] \arrow[d, shift right] & \arrow[l, shift left] \arrow[l, shift right]  \arrow[d, shift left] \arrow[d, shift right] \CG\nerv{3} &  \arrow[d, shift left] \arrow[d, shift right] \arrow[l, shift left, shift left] \arrow[l, shift right, shift right] \arrow[l] \CG\nerv{4} & \arrow[l, shift left, shift left, shift left] \arrow[l, shift right, shift right, shift right] \arrow[l, shift left] \arrow[l, shift right] \cdots  \\
    \CG\nerv{1} & \arrow[l, shift left] \arrow[l, shift right] \CG\nerv{2} & \arrow[l, shift left, shift left] \arrow[l, shift right, shift right] \arrow[l] \CG\nerv{3} &\arrow[l, shift left, shift left, shift left] \arrow[l, shift right, shift right, shift right] \arrow[l, shift left] \arrow[l, shift right] \cdots  \\
\end{tikzcd} 
\right).
\]
\end{definition}
\begin{remark}
The columns and rows of the shift double groupoid are the nerves of groupoid structures. They actually come from the natural action of \( \CG \) on itself from the left and right, respectively. We usually think of the rows and columns as being indexed by natural numbers starting at \( 0 \). Therefore, \(0\)-th column corresponds to the nerve of the (left) action groupoid \( \CG \ltimes \CG \grpd \CG \) and the \( n \)-th column corresponds to the nerve of the action groupoid \( \CG \ltimes \CG\nerv{n+1}\).
\end{remark} 
The shift double of a cosimplicial object is defined in a similar way. 
\begin{definition}
Given a cosimplicial vector space \( V \) the \text{shift double} of \( V \), denoted \( V_{\shift} \), is obtained by arranging the left and right shifts into the rows and columns (respectively) of a double structure as below:
\[ 
V_{\shift} := \left( \begin{tikzcd}[arrows=<-]
      \vdots \arrow[d, shift left, shift left, shift left] \arrow[d, shift right, shift right, shift right] \arrow[d, shift left] \arrow[d, shift right]  &                                               \vdots \arrow[d, shift left, shift left, shift left] \arrow[d, shift right, shift right, shift right] \arrow[d, shift left] \arrow[d, shift right]  &                                                                                    \vdots \arrow[d, shift left, shift left, shift left] \arrow[d, shift right, shift right, shift right] \arrow[d, shift left] \arrow[d, shift right] & \iddots \\
    V\nerv{3} \arrow[d, shift left, shift left] \arrow[d, shift right, shift right] \arrow[d] & \arrow[l, shift left] \arrow[l, shift right] V\nerv{4} \arrow[d, shift left, shift left] \arrow[d, shift right, shift right] \arrow[d] & \arrow[l, shift left, shift left] \arrow[l, shift right, shift right] \arrow[l] V\nerv{5} \arrow[d, shift left, shift left] \arrow[d, shift right, shift right] \arrow[d] & \arrow[l, shift left, shift left, shift left] \arrow[l, shift right, shift right, shift right] \arrow[l, shift left] \arrow[l, shift right] \cdots  \\
    V\nerv{2}  \arrow[d, shift left] \arrow[d, shift right] & \arrow[l, shift left] \arrow[l, shift right]  \arrow[d, shift left] \arrow[d, shift right] V\nerv{3} &  \arrow[d, shift left] \arrow[d, shift right] \arrow[l, shift left, shift left] \arrow[l, shift right, shift right] \arrow[l] V\nerv{4} & \arrow[l, shift left, shift left, shift left] \arrow[l, shift right, shift right, shift right] \arrow[l, shift left] \arrow[l, shift right] \cdots  \\
    V\nerv{1} & \arrow[l, shift left] \arrow[l, shift right] V\nerv{2} & \arrow[l, shift left, shift left] \arrow[l, shift right, shift right] \arrow[l] V\nerv{3} &\arrow[l, shift left, shift left, shift left] \arrow[l, shift right, shift right, shift right] \arrow[l, shift left] \arrow[l, shift right] \cdots  \\
\end{tikzcd} \right) . 
\]
\end{definition}
\subsection{Shift double complexes}
If \( V \) is a cosimplicial vector space and \( V_{\shift} \) is its associated shift double, then each row and column of \( V_{\shift} \) is a cosimplicial vector space and has an associated cochain complex.

\begin{definition}
If \( k \) is a natural number then the \( k\)-th (left) shifted differential is defined to be the cosimplicial differential for the \( k \)th shifted complex \( V_L[k] \). More concretely, for \( n \ge 0 \) we have:
\[ \delta_L[k]^n := \sum_{i=0}^{n+1} (-1)^i \phi^{k+n+1}_i \colon V\nerv{n+m} \to V\nerv{n+m+1}. \]
Similarly, the right shifted differential is:
\[ \delta_R[k]^n := \sum_{i=0}^{n+1} (-1)^i \phi^{k+n+1}_{k+i} \colon V\nerv{n+m} \to V\nerv{n+m+1}.\]
\end{definition}
\begin{definition}\label{defn:shift.double.complex}
Let \( V \) be a cosimplicial vector space. The \emph{shift double complex} of \( V \), denoted by \( \Ch(V)_{\shift} \), is the one obtained by taking the horizontal maps to be the left shifts of the simplicial differential and vertical maps to be the right shifts of the simplicial differential as illustrated below:
    \[ \Ch(V)_{\shift} := \left( 
\begin{tikzcd}[column sep=large, row sep = large]
 \vdots &  \vdots &  \vdots &  \iddots \\
V^{(3)}  \arrow[r , "{\delta_L[3]}"] \arrow[u, "{\delta_R [1]}"] & V^{(4)}  \arrow[r , "{\delta_L[3]}"] \arrow[u, "{\delta_R [2]}"] & V^{(5)}  \arrow[r , "{\delta_L[3]}"] \arrow[u, "{\delta_R [3]}"] & \cdots \\
V^{(2)}  \arrow[r , "{\delta_L[2]}"] \arrow[u, "{\delta_R [1]}"] & V^{(3)}  \arrow[r , "{\delta_L[2]}"] \arrow[u, "{\delta_R [2]}"] & V^{(4)}  \arrow[r , "{\delta_L[2]}"] \arrow[u, "{\delta_R [3]}"] & \cdots \\
V^{(1)}  \arrow[r , "{\delta_L[1]}"] \arrow[u, "{\delta_R [1]}"] & V^{(2)}  \arrow[r , "{\delta_L[1]}"] \arrow[u, "{\delta_R [2]}"] & V^{(3)}  \arrow[r , "{\delta_L[1]}"] \arrow[u, "{\delta_R [3]}"] & \cdots \\
\end{tikzcd}\right).
\]
We think of this double complex as being defined for non-negative degree. When indexing this structure we take the first index to indicate the row and the second index to indicate the column so:
\[V_{\shift}\nerv{n,m} := V\nerv{n+m+1} \]
\end{definition}
\begin{example}
    Let \( \CG \) be a Lie groupoid. Let \( \CA \) denote the associated cosimplicial algebra of functions on the nerve of \( \CG \). Then \( \Ch(\CG)_{\shift} \) is the double complex obtained by applying the ``groupoid cochain complex'' functor to the shift double of \( \CG \). We will refer to this as the \emph{shift double complex of \( \CG \)}.
\end{example}

\section{Retractions of (co)simplicial structures}\label{sec:augumentations and retracts}
Many of our arguments will rely on proving that a given left/bottom augmentation is acyclic. This will be primarily accomplished by constructing a homotopy operator that exhibits the fact that all cohomology groups vanish.

\subsection{Retractions along augmentations}
In the simplicial context, the homotopy operators we will construct correspond to the notion of a retraction. Specifically, a retraction along an augmentation of the simplicial structure.
\begin{definition}
Suppose \( X \) is a simplicial manifold and \( p \colon X\nerv{0} \to Y \) is an augmentation. A \emph{left retraction} of \( X \) along \( P \) consists of the following data:

A collection of smooth maps \( \{ r^n \}_{n \in \mathbb{N}} \) where \( r^0 \colon Y \to X\nerv{0}  \) and for \( n \ge 1 \):
    \[ r^n \colon X\nerv{n-1} \to X\nerv{n} \]
    and which satisfies the following properties:
\begin{itemize}
    \item \( r \) intertwines \( d^1_1 \) and \( p \): 
    \[  d^1_1 \circ r^1 = r^0 \circ p ,   \]
    \item for all \( n \ge 1 \):
    \[ d^{n}_0 \circ r^n = \Id ,\]
    \item \( r \) ``shifts'' the face maps to the left\footnote{By shifting to the left we mean that it ``decreases'' the index \( i \) by one. We imagine the face maps \( d^n_i \) as being arranged in increasing order from left to right.}. For all \( 1 \le i \le n + 1\):
    \[ d^{n+1}_i  \circ r^{n+1} = r^n \circ  d^{n}_{i-1}. \]
\end{itemize}
\end{definition}
Left retracts are actually homotopy equivalences between the original simplicial set \( X \) and \( Y \) thought of as a trivial simplicial set. The only difference from the standard theory of simplicial sets is that we are working in the smooth category and so we require that the maps involved are smooth. 

We will not spend much time discussing general homotopy equivalences between simplicial sets formally, however, and we will keep our discussion grounded by speaking in terms of these explicit retractions.

\subsection{Examples of retractions}
We are mainly interested in the nerve of a Lie groupoid. Therefore, our canonical example of a left retract of a Lie groupoid comes from a ``principal'' groupoid.
\begin{definition}
    A Lie groupoid \( \CG \) is \emph{principal} if the isotropy groups are trivial and the orbit space \( B \) admits a smooth structure which makes the projection \( p: \CG_0 \to B \) a submersion.
\end{definition}
Principal groupoids appear frequently in the study of Lie groupoids. They have a few equivalent names and characterizations. We use the term principal because they are precisely the groupoids for which the natural action of \( \CG \) on \( \CG_0 \) constitutes a principal groupoid bundle.
\begin{lemma}\label{example:principal.groupoid.retraction}
Suppose \( \CG \) is a principal Lie groupoid. Suppose  there exists a global section \( \sigma \colon B \to \CG_0 \) of \( p \). Then there exists a left retract of \( \CG \) to \( B \) along \( p \).
\end{lemma}
    \begin{proof}
    The proof is by a direct construction. Let us define a family of maps as follows:
\[ r^0 := \sigma \colon B \to \CG_0 \  \]
\[ r^1 \colon \CG_0 \to \CG_1  \] 
defined uniquely by the equations:
    \[ \bt \circ r^1 = \Id , \qquad \bs \circ r^1 = r^0 \circ p. \]
For \( n \ge 2\) we take:
    \[ r^{n}(g_{n-1},\ldots, g_1) := \left(g_{n-1} , \ldots , g_1 , (g_{n-1} \cdots g_1)^{-1} \cdot (r^1 \circ t)(g_{n-1}))\right). \]
    We claim that this data constitutes a left retract of \( \CG \) to \( B \) along \( p \).

    First let us explain why \( r^1 \) is well defined. Since the isotropy groups of \( \CG \) are trivial, an arrow in \( \CG_1 \) is completely determined by its source and target. Furthermore, since \( r^0 \) is a section of \( p \), for any point \( x \in \CG_0 \) there is a unique arrow in \( \CG \) with source \( r^0 \circ p(x) \) and target \( x \). And the smoothness of the map $r^0$ follows from the uniqueness of its choice. 

    It is not difficult to verify that \( d^n_0 \circ r^n = \Id \). For the case \( n = 1 \) we have this by construction of \( r^1 \). To show that this is a left retract we must show that for all \( 1 \le i \le n \) we have:
    \[ d^n_i \circ r^n  = r^{n-1} \circ d^n_{i-1}. \]
    
    The above equation is fairly straightforward to verify for \( i < n\). To show the case \( i = n \), it is helpful to first point out that the map \( r^1 \) has the following property for all \( x \in \CG_0 \) and \( h \in \CG \) with \( \bs(h) = x \):
    \[ h \cdot r^1(x) = r^1(t(h)). \]
    With this fact in mind, we get the following computation for all \( (g_n,\ldots , g_1) \in \CG\nerv{n-1} \):
    \begin{align*}
        d^{n}_{n} \circ r^n (g_{n-1},\ldots , g_1) 
        = d^n_n&\left(g_{n-1},\ldots , g_1, (g_{n-1} \cdots g_1)^{-1} \cdot r^1(\bt(g_{n-1}))\right) \\
        = &\left(g_{n-2}, \ldots , g_1, (g_{n-1} \cdots g_1)^{-1} \cdot r^1(\bt(g_{n-1})) \right)  \\
        = &\left(g_{n-2}, \ldots , g_1, (g_{n-2} \cdots g_1)^{-1} \cdot  g_{n}^{-1} \cdot r^1(\bt(g_n))\right) \\
        = &\left(g_{n-2}, \ldots , g_1, (g_{n-2} \cdots g_1)^{-1}  \cdot r^1(\bt(g_{n-2}))\right)\\
        = &r^{n-1} \circ d^{n-1}_{n-1} (g_{n-1}, \ldots , g_1)  .
    \end{align*} 
\end{proof}

\begin{example}\label{example:standard.left.retract}
    Suppose \( \CG \) is a Lie groupoid and let \(  k \ge 1 \) be a natural number. Consider the right shift \( \CG_R[k] \). Associated to this right shift, there is a natural projection map \( d^k_k \colon \CG\nerv{k} \to \CG\nerv{k-1} \) which can easily be seen to be invariant with respect to the source and target of \( \CG_R[k] \). 
    
    Furthermore, there is a natural left retract of \( \CG_R[k] \) to \( \CG\nerv{k-1} \) along \( d^k_k\) given by:
    \[   \{ s^{k+n-1}_k \colon \CG\nerv{k+n-1} \to \CG\nerv{k+n} \}_{n \in \mathbb{N}} \]
    That this data constitutes a left retract follows directly from the axioms regarding how face maps and degeneracy maps interact in a simplicial set. We call this the \emph{standard retract} of \( \CG_R[k] \) to \( \CG\nerv{k-1} \).
\end{example}

\subsection{Left retracts of cosimplicial vector spaces}
Let us now take a look at the left retracts in the context of cosimplicial vector spaces.
\begin{definition}
    Let \( V \) be a cosimplicial vector space and suppose we are given an augmentation \( L \colon W \to V\nerv{0} \). A \emph{left retract} \(  \{\theta^n\}_{n \in \mathbb{N}} \) of \( V \) along \( L \) consists of a family of linear maps \(\{ \theta^n \}_{n \in \mathbb{N}} \) where: 
    \[ \theta^0 \colon V\nerv{0} \to W \]
    and for \( n \ge 1 \) we have:
\[   \theta^n  \colon V\nerv{n} \to V\nerv{n-1} \]
and which satisfies the following equations for all \( n \ge 0\):
\begin{itemize}
    \item \( \theta \) intertwines \( \phi^1_1 \) and \( L \): 
    \[ \theta^1 \phi^1_1 = L \theta^0 ,  \]
    \item for all \( n \ge 1 \), \( \phi^n_0 \) is a left inverse of \( \theta^{n}\):
    \[ \theta^{n} \phi^{n}_0 = \Id ,\]
    \item \( \theta \) ``shifts'' the face maps to the left. For all \( 1 \le i \le n + 1\):
    \[ \theta^{n+1} \phi^{n+1}_i = \phi^{n}_{i-1} \theta^n . \]
\end{itemize}
\end{definition}
\begin{example}
    The notion of a left retract for cosimplicial vector spaces is the dual to the notion of a left retract for Lie groupoids. Given an augmentation of a Lie groupoid \( p \colon \CG\nerv{0} \to Y \) and a left retraction \( \{ r_n \}_{n \in \mathbb{N}} \) along \( p \) then \( \{ r_n^* \}_{n \in \mathbb{N} }\) constitutes a left retraction of \( \Ch(\CG) \) along \( p^* \colon C^\infty(Y) \to C^\infty(\CG\nerv{0} )  \).
\end{example}
There is also a version of the standard left retract for shifted simplicial structures for cosimplicial vector spaces.
\begin{example}\label{example:example.standard.left.retract.vectorspace}
    Let \(V \) be a cosimplicial vector space and \( k \ge 1 \) be a natural number. Consider the right shift \(  V_R[k] \). The linear map \( \phi^k_k \colon V\nerv{k-1} \to V\nerv{k} \) is invariant for \( V_R[k] \). We can construct a left retract of \( V_R[k] \) to \( V\nerv{k-1} \) along \( \phi^k_k \) by taking the following maps:
    \[ \{ \sigma^{k+n-1}_k \colon  V\nerv{k+n} \to V\nerv{k+n-1} \}_{n \in \mathbb{N}} \]
    This construction is just the cosimplicial vector space analogue of Example~\ref{example:standard.left.retract}.
\end{example}
\subsection{Retractions and cohomology}
Augmentations and left retracts of cosimplicial vector spaces have cohomological implications as well. Left retractions are useful because they can be used to prove that augmentations are acyclic.
\begin{lemma}\label{lemma:trivial.simplicial.cohomology}
Suppose \( V \) is a cosimplicial vector space and \( L \colon W \to V \) is an augmentation. If there exists a left retract \( \{ \theta_n \} \) of \( W \) along \( L \) then \( L \) is acyclic.

In particular, for all \( n \ge 1 \) we have that:
\[ H^n(V) = 0 \]
and \( L_* \colon W \to H^0(W) \) is an isomorphism.
\end{lemma}
\begin{proof}
Recall that an acyclic augmentation is one for which the augmented complex:
\[
\begin{tikzcd}
    W \arrow[r,"L" ] & V\nerv{0} \arrow[r] & V\nerv{1} \arrow[r] & V\nerv{2} \arrow[r] & \cdots 
\end{tikzcd}
\]
has trivial cohomology groups in all degrees.

We will do this by showing that the \( \theta\)'s constitute a contracting homotopy of the augmented complex.

Using the definition of a left retract, a direct calculation (just a straightforward application of the definition) shows that we have the following three equations:

\begin{itemize}
    \item For \( n \ge 1 \) we have:
        \[ \theta^{n+1} \delta^n + \delta^{n-1} \theta^{n} = \Id ,\]
    \item At level \( n = 0 \) we have:
     \[ \theta^1 \delta^0 - L \theta^0 = \Id ,\]
     \item At level \( n = -1 \) we have:
      \[ \theta^0 L  = \Id .\]
\end{itemize}
These three equations immediately imply that all cohomology groups of the augmented complex vanish.
\end{proof}
\begin{corollary}
    Let \( \CG \) be a principal Lie groupoid and \( p \colon \CG_0 \to B \) be the projection to the orbit space. If \( p \) admits a global section, then the groupoid cohomology of \( \CG \) is concentrated in degree zero and \( p^* \colon C^\infty(B) \to  H^0(\CG)  \) is an isomorphism.
\end{corollary}
Actually, the above corollary is true even if a global section does not exist. The idea is that, even when a global section does not exist at the level of spaces, one can still use local sections together with a partition of unity to construct a retraction at the level of groupoid cochains. We will see this argument made more explicitly later (See Section~\ref{subsection:morita.invariance.groupoid.cohomology}).
\begin{corollary}
    Let \(V \) be a cosimplicial vector space and let \( V_R[k] \) be a right shift. Then the map \( \phi^k_k \colon V\nerv{k-1} \to V\nerv{k} \) constitutes an acyclic augmentation of \( V_R[k] \).
    
    In particular, the cohomology of \( V_R[k] \) is concentrated in degree zero and the map \( \phi^k_k \colon V\nerv{k-1} \to  H^0(V_{R}[k]) \) is an isomorphism.
\end{corollary}
\subsection{Cohomology of shift double structures}
We can now compute the cohomology of shift double structures. The last observation of the previous subsection was that the cohomology of a shift of a cosimplicial vector space is concentrated in degree zero.

More specifically, every shift comes with a natural acyclic augmentation. We can put all of these augmentations together to construct acyclic augmentations of the full double structure.

If \( V \) is a cosimplicial vector space. There are two natural inclusions of the cochain complex of \( V \) into its shift double complex.
\begin{definition}
    Let \( V \) be a cosimplicial vector space. The \emph{standard left augmentation}, denoted 
    \[ L_V \colon \Ch(V) \to \Ch(V)_{\shift} \] 
 is given by the face maps:
    \[ \left\{ \phi^{n+1}_{0} \colon V\nerv{n} \to V\nerv{n+1} = V_{\shift}\nerv{n,0} \right\}_{n \in \mathbb{N}}.\]
    Similarly, the \emph{standard bottom augmentation}, denoted 
    \[ B_V \colon \Ch(V) \to \Ch(V)_{\shift} \]
    is given by the face maps:
    \[  \{ \phi^{n+1}_{n+1} : V\nerv{n} \to V\nerv{n+1} = V_{\shift}\nerv{0,n} \}_{n \in \mathbb{N}}.  \]
    These two maps are illustrated below
    \[
\begin{tikzcd}[column sep=large, row sep = large]
\vdots & \vdots &  \vdots &  \vdots &  \iddots \\
V\nerv{2} \arrow[u, "\delta"] \arrow[r, dashed, hook, "\phi^3_0"] & V^{(3)}  \arrow[r , "{\delta_L[3]}"] \arrow[u, "{\delta_R [1]}"] & V^{(4)}  \arrow[r , "{\delta_L[3]}"] \arrow[u, "{\delta_R [2]}"] & V^{(5)}  \arrow[r , "{\delta_L[3]}"] \arrow[u, "{\delta_R [3]}"] & \cdots \\
V\nerv{1} \arrow[u, "\delta"] \arrow[r, dashed, hook, "\phi^2_0"] & V^{(2)}  \arrow[r , "{\delta_L[2]}"] \arrow[u, "{\delta_R [1]}"] & V^{(3)}  \arrow[r , "{\delta_L[2]}"] \arrow[u, "{\delta_R [2]}"] & V^{(4)}  \arrow[r , "{\delta_L[2]}"] \arrow[u, "{\delta_R [3]}"] & \cdots \\
V\nerv{0} \arrow[u, "\delta"] \arrow[r, dashed, hook, "\phi^1_0"] & V^{(1)}  \arrow[r , "{\delta_L[1]}"] \arrow[u, "{\delta_R [1]}"] & V^{(2)}  \arrow[r , "{\delta_L[1]}"] \arrow[u, "{\delta_R [2]}"] & V^{(3)}  \arrow[r , "{\delta_L[1]}"] \arrow[u, "{\delta_R [3]}"] & \cdots \\
& V\nerv{0} \arrow[u, dashed, hook, "\phi^1_1"] \arrow[r, "\delta"] & V\nerv{1} \arrow[u, dashed, hook, "\phi^2_2"] \arrow[r, "\delta"] & V\nerv{2} \arrow[u, dashed, hook, "\phi^3_3"] \arrow[r,"\delta"] & \cdots .
\end{tikzcd}
    \]
\end{definition}
It turns out that these left and bottom augmentations are actually quasi-isomorphisms. This means that the shift double structure can be thought of as an ``exploded'' model for the original object.

Let us state our claim more clearly:
\begin{lemma}\label{lemma:unit.is.retract.of.shift}
The bottom augmentation
\[B_V \colon \Ch(V) \to \Ch(V)_{\shift} \] 
and left augmentation
\[L_V \colon \Ch(V) \to \Ch(V)_{\shift}  \]
are quasi-isomorphisms.
\end{lemma}
\begin{proof}
    We prove this using the augmentation lemma. We saw earlier in Example~\ref{example:example.standard.left.retract.vectorspace} that one can construct a left retraction of the augmented complexes associated to each column. We also learned that this implies that the augmented rows are acyclic. By the Augmentation Lemma (Lemma~\ref{lemma:augmentation.lemma}), it follows that \( B_V \) is a quasi-isomorphism.

    A symmetrical argument can be applied to \( L_V \). The only difference is that, technically, one must use ``right retractions'' rather than left which we have not formally defined or constructed. However, this is only a matter of adjusting the indexing in our definitions.
\end{proof}

\section{Target retractions over open sets}\label{sec:shifted retractions of groupoids}
Briefly, a target retraction is a special kind of left retraction that takes place over an open subset \( U \subset \CG\nerv{0} \) of the objects of a Lie groupoid.

Notice that gave explicit formulas for constructing left retractions of the right shifts of a Lie groupoid (See Example~\ref{example:standard.left.retract}) but, in general, there is no standard left retract for a left shift.

It turns out that it is still possible to construct left retracts of left shifts by choosing a section of the target map. Some complications in the following definitions come from the fact that it will be useful for us to be able to do this for an \emph{arbitrary local section} of the target map defined over open subsets of the objects. Retractions constructed in this way we will call ``target families.''

\subsection{Restricting left shifts to open sets}
Note that the left shift \( \CG_L[1]\) of a Lie groupoid is actually a principal groupoid with orbit space \( \CG\nerv{0} \). The projection to the orbit space is given by the target map: \( t \colon  \CG\nerv{1} \to \CG\nerv{0} \).

Given an open subset \( U \subset \CG\nerv{0} \) it makes sense to restrict the left shift to the open set \( U \) by restricting to all arrows that lie ``above'' \( U \).

Towards that end, we have the following definition:
\begin{definition}
    Let \( \CG \) be a Lie groupoid throughout this section. Given an open set \( U \subset \CG_0 \) we write \( U\nerv{0} := U \) and for \( n \ge 1 \):
\[ U\nerv{n} := \{ (g_n, \ldots , g_1) \ : \ t(g_n) \in U \} \subset \CG\nerv{n} \]
\end{definition}
\begin{remark} Warning: The nerve notation is suggestive here but the collection \( \{ U\nerv{n} \} \) does not quite constitute a simplicial submanifold of the nerve of \(\CG \). The reason for this is that not all of the face maps are defined for \( U\nerv{n} \). In general, the image of \(d^n_n |_{U\nerv{n}} \) is not contained in \( U\nerv{n-1} \).
\end{remark}
Although this collection does not constitute a simplicial submanifold of \( \CG \). We can use this family to construct subgroupoids of the shifts of \( \CG \).
\begin{definition}
Let:
\[ U[k]\nerv{n} := \left\{ U\nerv{n+k} \right\}_{n \in \mathbb{N}}. \]
We equip \( U[k]\nerv{\bullet}\) with the structure of a simplicial manifold by letting it inherit its face maps from \( \CG[k]\).
In other words we take:
\[ d^n_i |_{U\nerv{n}} \colon U\nerv{n} \to U\nerv{n-1} . \]
\end{definition}
This is well defined because the only face map that does not restrict to the family \( d^n_i \colon U\nerv{n} \to U\nerv{n-1} \) is the ``top'' face map where \( i = n \). Since any left shift of \( \CG \) omits this face map, one obtains a well defined simplicial manifold structure\footnote{Actually the simplicial manifold \( U[k] \) is the nerve of a Lie groupoid. It is the nerve of the (right) action groupoid of \( \CG \) on \( U\nerv{k} \). This groupoid is principal and has orbit space \( U\nerv{k-1} \).} on \( U[k] \).

If \( \{ U_a \}_{a \in A } \) is an open cover of \( \CG_0 \) then for each \( k \ge 1 \) this construction provides us with an ``simplicial'' open cover \( \{ U_a[k] \}_{a \in A} \) of \( \CG_L[k] \).
\subsection{Target families}
Roughly speaking, a target family is a special kind of left retract of the local left shift \( U[1] \) where \( U \) is an open subset of \( \CG \).

\begin{definition}
    Suppose \( \CG \) is a Lie groupoid and we are given an open set \( U \subset \CG_0 \) together with a section \( r^0 \colon U \to U\nerv{1}\) of the target map \( \bt \). 
    
    The associated \emph{target family over \( U \)} is the collection of maps:
    \[ \left\{ r^n \colon U\nerv{n} \to U\nerv{n+1} \right\}_{n \in \mathbb{N}} \]
    where for all \( n \ge 1 \):
    \[ r^n (g_{n} , \ldots , g_1 ) \mapsto \left(g_n,\ldots, g_1, \, (g_{n} \cdots g_1)^{-1} \cdot r^0(t(g_{n}))\right)  . \]   
\end{definition}
\begin{example}
    Consider the identity section \( u \colon \CG_0 \to \CG_1 \). This section has an associated target family where \( r^0 = u \) and for \( n \ge 1 \) we have:
    \[ r^n (g_n,\ldots,g_1) = (g_n,\ldots , g_1, (g_n \cdots  g_1)^{-1} ). \]
\end{example}
\begin{lemma}
    Let \( \CG \) be a Lie groupoid and let \( \{ r^n \colon U\nerv{n} \to \CG\nerv{n+1} \}_{n \in \mathbb{N}} \) be a target family. Then we have that for all \( n \ge 1\):
    \[ d^{n+1}_i \circ r^n = \begin{cases} \Id,& i = 0 ,\\ r^{n-1} d_{i-1}^n ,& 1 \le i \le n .\end{cases} \]
\end{lemma}
The statement of the lemma is essentially the proof as it follows from a rather straightforward computation. It tells us that any target family defines a left retract of the associated subgroupoids of the left shifts of \( \CG \). 
\begin{lemma}
    Suppose we have a target family \( \{ r^n \} \) over \( U \). Then the  collection \( \{ r^{k-1+n} \}_{n \in \mathbb{N}} \) constitutes a left retract of the groupoid \( U[k] \) along the map \( d^k_0 \colon U[k]\nerv{0} \to U\nerv{n-k}\).
\end{lemma}

\section{Groupoid homomorphisms}\label{sec:groupoid homomorphisms}
\subsection{Weak equivalences}\label{subset:weak equivalence}
\begin{definition}
    Let \( \CG \) and \( \CH \) be Lie groupoids. A groupoid homomorphism \( F \colon \CH \to \CG \) consists of a pair of smooth functions:
    \[ F_1 \colon \CH_1 \to \CG_1 , \qquad F_0 \colon \CH_0 \to \CG_0 \]
    with the properties that \( F(\bt(h)) = \bt(F_1(h)) \) and \( F_1(\bs(h)) = \bs(F_1(h)) \) for all \( h \in \CH_1  \) and:
    \[ \forall (h_2,h_1) \in \CH\nerv{2}, \qquad F_1(h_2 h_1) = F_1(h_2) F_1(h_1). \]
\end{definition}
\begin{definition}
    Suppose \( \CG \) is a Lie groupoid. A smooth map \( f \colon N \to \CG_0 \) is \emph{smoothly essentially surjective} if the smooth function:
    \[ \CG_1 \fiber{\bs}{f} N \to \CG_0 ,\qquad (g,x) \mapsto \bt(g) \]
    is a surjective submersion. 
    
    Given a function \( f \colon N \to \CG_0 \) which is smoothly essentially surjective the \emph{pullback groupoid} is the Lie groupoid \( f^! \CG \) where:
    \[ (f^! \CG)_0 := N ,\qquad (f^! \CG)_1 := N \fiber{f}{\bt} \CG_1 \fiber{\bs}{f} N . \]
    The source and target maps are:
    \[ \bs(x,g,y) := y ,\qquad \bt( x,g,y) := x ,\]
    and the groupoid multiplication is given by:
    \[ (x,g,y) \cdot (y,h,z) := (x,gh, z). \]
\end{definition}
\begin{definition}\label{defn:weak-equiv}
    A groupoid homomorphism \( F \) is called a \emph{weak equivalence} if \( F_0 \) is smoothly essentially surjective and the natural map:
    \[  F_! \colon \CH \to f^! \CG ,\qquad h \mapsto (\bt(h), F_1(h), \bs(h))  \]
    defines an isomorphism of Lie groupoids.
\end{definition}
\subsection{Double groupoid of a groupoid homomorphism}
\begin{definition}
    Given a groupoid homomorphism \( F \) one can construct a double simplicial structure \( \DG_F \) which we call the \emph{\( F\)-double groupoid}:
    \[
    \DG_F := \left( 
\begin{tikzcd}[column sep =small]
      \vdots \arrow[d, shift left, shift left, shift left] \arrow[d, shift right, shift right, shift right] \arrow[d, shift left] \arrow[d, shift right]  &                                               \vdots \arrow[d, shift left, shift left, shift left] \arrow[d, shift right, shift right, shift right] \arrow[d, shift left] \arrow[d, shift right]  &                                                                                    \vdots \arrow[d, shift left, shift left, shift left] \arrow[d, shift right, shift right, shift right] \arrow[d, shift left] \arrow[d, shift right] & \iddots \\
    \CG\nerv{3} \fiber{\bs}{F_0 \circ \bt} \CH\nerv{0} \arrow[d, shift left, shift left] \arrow[d, shift right, shift right] \arrow[d] & \arrow[l, shift left] \arrow[l, shift right] \CG\nerv{3} \fiber{\bs}{F_0 \circ \bt} \CH\nerv{1} \arrow[d, shift left, shift left] \arrow[d, shift right, shift right] \arrow[d] & \arrow[l, shift left, shift left] \arrow[l, shift right, shift right] \arrow[l] \CG\nerv{3} \fiber{\bs}{F_0 \circ \bt} \CH\nerv{2} \arrow[d, shift left, shift left] \arrow[d, shift right, shift right] \arrow[d] & \arrow[l, shift left, shift left, shift left] \arrow[l, shift right, shift right, shift right] \arrow[l, shift left] \arrow[l, shift right] \cdots  \\
    \CG\nerv{2} \fiber{\bs}{F_0 \circ \bt} \CH\nerv{0}  \arrow[d, shift left] \arrow[d, shift right] & \arrow[l, shift left] \arrow[l, shift right]  \arrow[d, shift left] \arrow[d, shift right] \CG\nerv{2} \fiber{\bs}{F_0 \circ \bt} \CH\nerv{1} &  \arrow[d, shift left] \arrow[d, shift right] \arrow[l, shift left, shift left] \arrow[l, shift right, shift right] \arrow[l] \CG\nerv{2} \fiber{\bs}{F_0 \circ \bt} \CH\nerv{2} & \arrow[l, shift left, shift left, shift left] \arrow[l, shift right, shift right, shift right] \arrow[l, shift left] \arrow[l, shift right] \cdots  \\
    \CG\nerv{1} \fiber{\bs}{F_0 \circ \bt} \CH\nerv{0} & \arrow[l, shift left] \arrow[l, shift right]\CG\nerv{1} \fiber{\bs}{F_0 \circ \bt} \CH\nerv{1}& \arrow[l, shift left, shift left] \arrow[l, shift right, shift right] \arrow[l]\CG\nerv{1} \fiber{\bs}{F_0 \circ \bt} \CH\nerv{2} &\arrow[l, shift left, shift left, shift left] \arrow[l, shift right, shift right, shift right] \arrow[l, shift left] \arrow[l, shift right] \cdots  \\
\end{tikzcd}
\right).
\]
In other words, we have 
\[ \DG_F\nerv{n,m} := \CG\nerv{n+1,m} \fiber{\bs}{F_0 \circ \bt} \CH\nerv{m} .  \]
For \( m \in \mathbb{N} \), each column \( \DG_F\nerv{\bullet,m}\) is the nerve of the (left) action groupoid:
\[ \CG \ltimes (\CG_1 \fiber{\bs}{F_0 \circ \bt} \CH\nerv{m} ) \grpd \CG_1 \fiber{\bs}{F_0 \circ \bt} \CH\nerv{m}  \]
\[ g_2 \cdot (g_1, h_m,\ldots , h_1) := (g_2 g_1, h_m,\ldots, h_1) .  \]
For \( n \in \mathbb{N} \), each column \( \DG_F\nerv{n,\bullet} \) is the nerve of the \emph{right} action groupoid:
\[ (\CG\nerv{n+1} \fiber{\bs}{F_0} \CH_0 ) \rtimes \CH \grpd  \CG\nerv{n+1} \fiber{\bs}{F_0} \CH_0\]
\[ (g_{n+1},\ldots , g_1 , x) \cdot h := (g_{n+1},\ldots, g_1 F_1(h), \bs(h)). \]
\end{definition}
The homomorphism \( F \) also defines natural maps from the \( F \)-double groupoid \( \DG_F \) into the shift double groupoid \( \CG_{{\shift}} \):
\[ F_{{\shift}} \colon \DG_F \to \CG_{{\shift}}, \]
\[ F_{{\shift}}\nerv{n,m} \colon \CG\nerv{n+1,m} \fiber{\bs}{F_0 \circ \bt} \CH\nerv{m}  \to \CG\nerv{n+m+1} , \]
\[ F_{{\shift}}\nerv{n,m} (g_{n+1},\ldots, g_1, h_m , \ldots, h_1) := (g_{n+1},\ldots, g_1, F_1(h_m),\ldots , F_1(h_1)) . \]
\subsection{Double complex of a groupoid homomorphism}
We saw that any Lie groupoid homomorphism has an associated double simplicial structure. Of course there is also a double complex.
\begin{definition}
    Let \( F \colon \CH \to \CG \) be a Lie groupoid homomorphism. Write \[ \DC_F\nerv{n,m} := C^\infty( \CG\nerv{n+1} \fiber{\bs}{F_0 \circ \bt} \CH\nerv{m}). \] 
    Using the simplicial differentials for each row and column of the \( F \)-double groupoid, we obtain a double complex \( \DC_F \) which we call the \emph{\( F\)-double complex}. 
    \[\DC_F := \left( 
\begin{tikzcd}[column sep=large, row sep = large]
 \vdots &  \vdots &  \vdots &  \iddots \\
\DC_F^{(2,0)}  \arrow[r , "\delta_F^H"] \arrow[u, "\delta_F^V"] & \DC_F^{(2,1)}  \arrow[r , "\delta_F^H"] \arrow[u, "\delta_F^V"] & \DC_F^{(2,2)}  \arrow[r , "\delta_F^H"] \arrow[u, "\delta_F^V"] & \cdots \\
\DC_F^{(1,0)}  \arrow[r , "\delta_F^H"] \arrow[u, "\delta_F^V"] & \DC_F^{(1,1)}  \arrow[r , "\delta_F^H"] \arrow[u, "\delta_F^V"] & \DC_F^{(1,2)}  \arrow[r , "\delta_F^H"] \arrow[u, "\delta_F^V"] & \cdots \\
\DC_F^{(0,0)}  \arrow[r , "\delta_F^H"] \arrow[u, "\delta_F^V"] & \DC_F^{(0,1)}  \arrow[r , "\delta_F^H"] \arrow[u, "\delta_F^V"] & \DC_F^{(0,2)}  \arrow[r , "\delta_F^H"] \arrow[u, "\delta_F^V"] & \cdots \\
\end{tikzcd}\right) .
    \]
\end{definition}
The \( F \)-double complex is related to the groupoid cohomology of \( \CG \) as well as the groupoid cohomology of \( \CH \). We begin by describing the relationship to the groupoid cohomology of \( \CH \).

There is a natural groupoid homomorphism \( P_F \) from the bottom row of the \( F \)-double groupoid to the nerve of \( \CH\):
\[
\begin{tikzcd}
    \CG\nerv{1} \fiber{\bs}{F_0 \circ \bt} \CH\nerv{0} \arrow[d, "P_F\nerv{0}"]
    &
    \arrow[l, shift left] \arrow[l, shift right]\CG\nerv{1} \fiber{\bs}{F_0 \circ \bt} \CH\nerv{1} \arrow[d, "P_F\nerv{1}"] &
    \arrow[l, shift left, shift left] \arrow[l, shift right, shift right] \arrow[l]\CG\nerv{1} \fiber{\bs}{F_0 \circ \bt} \CH\nerv{2} \arrow[d, "P_F\nerv{2}"]
    &\arrow[l, shift left, shift left, shift left] \arrow[l, shift right, shift right, shift right] \arrow[l, shift left] \arrow[l, shift right] \cdots  \\
    \CH\nerv{0} & \arrow[l, shift left] \arrow[l, shift right] \CH\nerv{1} & \arrow[l, shift left, shift left] \arrow[l, shift right, shift right] \arrow[l] \CH\nerv{2} & \arrow[l, shift left, shift left, shift left] \arrow[l, shift right, shift right, shift right] \arrow[l, shift left] \arrow[l, shift right] \cdots
\end{tikzcd}
\]
where:
\begin{equation}\label{defn:PF.definition}
    P_F\nerv{n} \colon \CG\nerv{1} \fiber{\bs}{F_0 \circ \bt} \CH\nerv{n} \to \CH\nerv{n} \qquad (g,h_n,\ldots,h_1) \mapsto (h_n,\ldots,h_1) .
\end{equation}
This groupoid homomorphism induces a bottom augmentation:
\[ 
\begin{tikzcd}
    \DC_F\nerv{0,0} \arrow[r] & \DC_F\nerv{0,1} \arrow[r] & \DC_F\nerv{0,2} \arrow[r] & \cdots \\
    \Ch(\CH)\nerv{0} \arrow[u, "P_F^*"] \arrow[r, "\delta"] & \Ch(\CH)\nerv{1} \arrow[u, "P_F^*"] \arrow[r, "\delta "] & \Ch(\CH)\nerv{2} \arrow[u, "P_F^*"] \arrow[r, "\delta "] &\cdots
\end{tikzcd}
\]
In particular, this means that there is a natural morphism in cohomology from \( H(\CH) \) to \( H(\DC_F) \). In fact, this is an isomorphism (a fact which we will prove shortly).

One can also relate the cohomology of \( \CG \) to the cohomology of \( \DC_F \) in the following way: recall that there is a morphism of double groupoids from the \( F \)-double groupoid to the shift double groupoid of \( \CG \) given by the maps:
\[ F_{{\shift}}\nerv{n,m} \colon  \CG\nerv{n+1} \fiber{\bs}{F_0 \circ \bt} \CH\nerv{m} \to \CG\nerv{n+m+1} \] 
\[(g_{n+1},\ldots, g_1, h_m, \ldots , h_1) \mapsto (g_{n+1},\ldots, g_1, F_1(h_m),\ldots , F_1(h_1) ). \]
This induces a morphism of double complexes:
\[ F_{{\shift}}^* \colon \Ch(\CG)_{{\shift}} \to \DC_F . \]
On the other hand, there are also the standard bottom and left augmentations:
\[ B^* \colon \Ch(\CG) \to \Ch(\CG)_{{\shift}}, \qquad L^* \colon \Ch(\CG) \to \Ch(\CG)_{{\shift}} \]
which we know are quasi-isomorphisms.

We will conclude this section with a lemma that summarizes the relationships between the various groupoid cohomologies and double complexes.
\begin{lemma}\label{lemma:diagram.lemma.algebras}
    Suppose \( F \colon \CH \to \CG \) is a Lie groupoid homomorphism. Then we have a commutative diagram of simplicial structures:
    \[
    \begin{tikzcd}
    \CG  & \arrow[l, "L", swap]    \CG_{{\shift}} \arrow[d, "B", swap]  & \arrow[l, "F_{{\shift}}", swap] \arrow[d, "P_F"]  \DG_F\\
       &   \CG   & \arrow[l, "F", swap] \CH 
    \end{tikzcd}
    \]
    where \( L \) and \( B \) are the standard left and bottom augmentations and \( P_F\) is the map defined in Equation~\ref{defn:PF.definition}.

    The morphisms \( L\), \( B \) and \( P_F \) admit left retractions. Therefore, the associated augmentations of complexes \( L^* , B^* \) and \( P_F^* \) are quasi-isomorphisms.
    \end{lemma}
\begin{proof}
    Lemma~\ref{lemma:unit.is.retract.of.shift} says that \( L^* \) and \( B^* \) are quasi-isomorphisms since they are the standard left and bottom augmentations of \( \Ch(\CG)_{\shift} \). The only part of this lemma that has not already been proved is the claim that  \( P_F\) admits a left retraction. 

    Let us fix \( m \in \mathbb{N} \). We will construct a left retract of the \( m \)-column of \( \DG_F \) to \( \CH\nerv{m} \) along \( P_F \). 
    
    For \( n \in \mathbb{N} \) let \( r^n \) be defined as follows:
\[ r^0 \colon \CH\nerv{m} \to \CG\nerv{1} \fiber{\bs}{F_0 \circ \bt} \CH\nerv{m} \]
\[ (h_m,\ldots,h_1) \to (\bu \circ F_0 \circ \bt(h_m), h_m, \ldots , h_1) \]
and for \( n \ge 1 \):
\[ r^n \colon \CG\nerv{n} \fiber{\bs}{F_0 \circ \bt} \CH\nerv{m} \to \CG\nerv{n+1} \fiber{\bs}{F_0 \circ \bt} \CH\nerv{m} \]
\[(g_n,\ldots , g_1 , h_m,\ldots , h_1) \mapsto (g_n, \ldots , g_1 , \bu \circ \bs(g_1) , h_m, \ldots , h_1) \]
To conclude that this is a left retract we must show that:
\begin{itemize}
    \item compatibility between \( r \) and \( p \):
    \[ d^1_1 \circ r^1 = r^0 \circ p ;\]
    \item for all \( n \ge 1 \):
    \[ d^n_0 \circ r^n = \Id;\]
    \item for all \( 1 \le i \le n+1 \):
    \[ d^{n+1}_i \circ r^{n+1} = r^n \circ d^n_{i-1}, \]
\end{itemize}
where the maps \( \{ d^n_i \} \) are the face maps for the \( m \)-column of \( \DG_F \). However, these face maps are precisely the same as the face maps inherited from the \( \CG \)-part of \( \DG_F \).

For example, computing the LHS of the equation from the first bullet point gives:
\begin{align*} 
d^1_1 \circ  r^1 (g_1,h_m , \ldots , h_1)) &= d^1_1 (g_1, \bu \circ F_0 \circ \bt(h_m) , h_m , \ldots , h_1) \\ 
&= (  \bu \circ F_0 \circ \bt(h_m) , h_m, \ldots , h_1) \\
&=  r^0 \circ p (g_1, h_m, \ldots , h_1)   
\end{align*}
The second bullet point will hold as a consequence of the right unit axiom. The third bullet point will hold as a consequence of the fact that these fact maps do not ``touch'' the inserted unit. We leave it to the reader to verify the details of these equations (it is just a direct calculation from the definition).
\end{proof}

\section{Cosimplicial Modules over groupoids}\label{sec:cosimplicial modules over groupoids}
\subsection{Sheaves of modules}
Suppose \( M \) is a smooth manifold. We write \( C^\infty_M \) to denote the sheaf of functions on \( M \). 
\begin{definition}
A sheaf of modules \( \CE \) on a manifold \( M\) consists of a sheaf of \( C^\infty_M \)-modules \( \CE \). Given a smooth map \( f \colon N \to M  \) the \emph{pullback} \( f^* \CE \) of \( \CE \) along \( f \) is sheaf:
    \[ U \mapsto \CE(M) \underset{C^\infty_M(M)}{\otimes} C^\infty_N(U) \]
    where \( C^\infty_M(M) \) acts on \( C^\infty_N(U) \) along the homomorphism \( \alpha \mapsto f^* \alpha |_U \) and we take the tensor product \( \otimes \) to be taken over \emph{locally} finite sums/ 
     In other words, elements of \( f^*\CE( U) \) are represented by formal sums of the form:
    \[ \sum_{a \in A} e_a \otimes \rho_a \]
    where \( \rho_a \) is a locally finitely supported family of functions on \( U \).
\end{definition}
\begin{remark}Our definition of the pullback implicitly takes advantage of some particularly nice properties of sheaves of modules on manifolds. A standard argument using partitions of unity verifies that the pullback operation we have defined indeed results in a sheaf. 
\end{remark}
\begin{example}
    Suppose \( \CE \) is a sheaf of modules on a manifold \( M \) and let \(\iota \colon U \to M \) be the inclusion of an open subset. Then the pullback sheaf \( \CE|_U := \iota^* U \) has global sections \( \iota^* \CE (U) = \CE(U)  \).
\end{example}
\begin{example}
    Suppose \( E \to M \) is a vector bundle. Then \( \Gamma_E \) the sheaf of sections of \( E \) is a sheaf of modules on \( M \). If \( f \colon N \to M \) is a smooth map then \( f^* \Gamma_E \) is canonically isomorphic to \( \Gamma_{f^* E}\) where \( f^* E  = E \fiber{\pi,f} N \to N \) is the usual pullback bundle.
\end{example}
\begin{definition}
    Suppose \( \CE \) is a sheaf of modules on \( M \) and \( \CW \) is a sheaf of modules on \( N \). A \emph{morphism} from \((F,f) \colon \CE \to \CW \) consists of a smooth map \( f \colon M \to N \) and a morphism of \( C^\infty_M \)-modules:
    \[ F \colon \CE \to f^* \CW .\]
    A \emph{comorphism} \((F,f) \colon \CW \coto \CE\) consists of a smooth map \( f \colon M \to N \) and a morphism of \( C^\infty_M\)-modules:
    \[ F \colon f^*\CW \coto \CE.\]
\end{definition}
\begin{example}
    Suppose \( E \to M \) and \( W \to N \) are vector bundles. If \( F \colon E \to W \) is a vector bundle map over \( f \colon M \to N \). One can factor a bundle map through the pullback bundle by defining a pushforward operation:
    \[ F_* : E \to f^* W   \]
    There is a canonical isomorphism between \( \Gamma_{f^* W} \) and the pullback sheaf of modules \( f^* \Gamma_W\). Hence, this yields a morphism of sheaves of modules 
    \[ (F_* , f) \colon \Gamma_E \to \Gamma_W  \]

    On the other hand, if we are given a smooth map \( f \colon M \to  N \) and a morphism of vector bundles \( G \colon f^*W \to E \) covering the identity, then one obtains a natural comorphism of modules \( E \to M \)
    \[ (G^*,f) \colon \Gamma_W \coto \Gamma_E \]
\end{example}
\begin{example}
    Suppose \( f \colon M \to N \) is a smooth map. The sheaves of \( k \)-forms \( \Omega^k_M\) and \( \Omega^k_N \) are sheaves of modules on \( M \) and \( N \). The usual pullback \( (f^*,f) \colon \Omega^k_M \coto \Omega^k_N \) is a comorphism.
\end{example}
\subsection{Good cosimplicial modules}
Not all cosimplicial modules will be well behaved with respect to weak equivalences. We will focus on a subset of them which we call ``good'' cosimplicial modules.

The property of being good can be thought of as a kind of Kan condition for the associated cosimplicial vector space. It's slightly stronger than the usual Kan conditions however and one can roughly think of it as saying: ``Kan filling maps of the groupoid can be lifted to Kan filling maps of the cosimplicial module in a way compatible with face maps.''
\begin{definition}\label{definition:cosimplicial.module}
    Suppose \( \CG \) is a Lie groupoid. A \emph{cosimplicial module} on \( \CG \) is a cosimplicial object in the category of modules with comorphisms. In other words, it consists of the following data:
    \begin{itemize}
        \item for each natural number \( n \ge 0\), a sheaf of modules \( \CE\nerv{n} \) on \( \CG\nerv{n} \).
        ,
        \item for each face map \( d^n_i \), a comorphism of modules 
        \[(\phi^n_i, d^n_i)  \colon \CE\nerv{n-1} \coto \CE\nerv{n},\]
        \item for each degeneracy map \( s^n_i \), a comorphism of modules 
        \[ (\sigma^n_i, s^n_i)  \colon \CE\nerv{n+1} \coto \CE\nerv{n} .\]
    \end{itemize}
    (Goodness) A cosimplicial module is \emph{good} if given any local section \( r^0 \colon U \to \CG_1 \) of the target map with associated target family \( \{ r^n \}_{n \in \mathbb{N}} \), there exists a family of comorphisms 
    \[ \left\{ ((r^n)^\#, r^n)  \colon \CE\nerv{n+1}|_{U\nerv{n+1}} \coto \CE\nerv{n}|_{U\nerv{n}} \right\}_{n \in \mathbb{N}} \]
    which satisfies the target family relations:
\[ (r^n)^\# \phi^{n+1}_i = \begin{cases}
    \Id, & i = 0, \\ \phi^n_{i-1} (r^{n-1})^{\#}, & 1 \le i \le n .
\end{cases}
\]
\end{definition}

At this time we are not aware of a natural example of a cosimplicial module over a groupoid which is not ``good.'' Any module which is sufficiently functorial is easily seen to be good.
\begin{example}
    Suppose \( M \) is a smooth manifold, thought of as a trivial groupoid \( M \grpd M \). Then the face and degeneracy maps are all just the identity. In this case, any cosimplicial module on \( M \grpd M \) will be good.
\end{example}
\begin{example}
    Let \( \CG \) be a Lie groupoid and let \( k \ge 0\) be an integer. Consider the sheaves of differential forms on the nerve \(\{ \Omega_{\CG\nerv{n}}^k \} \) then we we have coface maps:
    \[ ( (d^n_i)^* , d^n_i) \colon \Omega_{\CG\nerv{n-1}}^{k} \coto \Omega_{\CG\nerv{n}}^k  \] 
    which are the coface maps for a cosimplicial structure.
    
    It is also clearly good since the functorial nature of this sheaf means any target family can certainly be lifted to the level of modules.
\end{example}
\begin{example}\label{example:representation.as.cosimplicial.module}
    Let \( \CG \) be a Lie groupoid and suppose we have a vector bundle \( \pi \colon E \to \CG_0 \) together with a (right) representation of \( \CG \) on \(E \).

    One can construct a cosimplicial module by considering the action groupoid \( E \rtimes \CG \grpd E \). Note that each level of the nerve of the action groupoid is of the form \(E \fiber{\pi}{\bt} \CG\nerv{n} \) which is a vector bundle over \( \CG\nerv{n} \). The face maps of the nerve of \( \CG \ltimes E \grpd E \) are fiber-wise isomorphisms so they can be equivalently thought of as morphisms or comorphisms. Therefore, if we take:
    \[ \CE\nerv{n} := \Gamma( E \fiber{\pi}{\bt} \CG\nerv{n} ) \]
    for each face map \( d^n_i \) of \( \CG\nerv{n} \) one has a pullback map:
    \[ \phi^n_i \colon \CE\nerv{n-1} \to \CE\nerv{n}. \]
    Hence we obtain a cosimplicial module over \( \CG \).
\end{example}
\begin{lemma}
    The simplicial module in Example~\ref{example:representation.as.cosimplicial.module} is good.
\end{lemma}
\begin{proof}
    Suppose \( r^0 \colon U \to \CG_1 \) is a section of the target map and \( \{ r^n \}_{n \in \mathbb{N}} \) is the associated target family. Now consider the bundle maps:
    \[ R^n \colon E \fiber{\pi}{\bt} U\nerv{n}  \to E\fiber{\pi}{\bt} U\nerv{n+1} , \qquad R^n(e,g_n,\ldots , g_1)) := (e,r^n(\pi(e))) . \]
    These fiberwise linear maps cover the target family on \( \CG \) and one sees easily that they define a target family of the groupoid \( E \rtimes \CG \) over \( E|_U \). Since they are fiberwise isomorphisms, they define pullbacks at the level of sections and we define \( (r^n)^\# \colon \CE\nerv{n+1} \to \CE\nerv{n}\) to be these pullback maps.
\end{proof}
\subsection{Cohomology of cosimplicial modules}
Since the global sections of a cosimplicial module constitute a cosimplicial vector space. They have a natural cohomology theory.
\begin{definition}
    Let \( \CE \) be a cosimplicial module on \( \CG \). The global sections of \( \CE \) form a cosimplicial vector space. Then the \emph{cochain complex} of \( \CE \) is defined to be the natural cochain complex associated to this cosimplicial vector space. More explicitly, \( \Ch(\CE) \) is the graded vector space:
    \[ \Ch(\CE) := \bigoplus_{n \in \mathbb{N}} \CE\nerv{n}(\CG\nerv{n})  \]
    equipped with the degree one graded map 
    \[ \delta \colon \Ch(\CE) \to \Ch(\CE)  \]
    \[ \delta^n := \sum_{i=0}^{n+1} (-1)^i \phi^n_i  \]
\end{definition}
\subsection{Pullbacks of cosimplicial modules}

\begin{definition}
    Let \( \CG \) be a Lie groupoid and let \( \CE \) be a cosimplicial module on \( \CG \). If \( F \colon \CH \to \CG \) is a groupoid homomorphism then we obtain a pullback \( F^* \colon \Ch(\CG) \to \Ch(\CH) \) at the level of groupoid cochains.
    
    We define the pullback \( F^* \CE \) to be the cosimplicial module on \( \CH \) given by the pullbacks along the natural maps at the level of the nerves:
    \[ (F^* \CE)\nerv{n} := \CE\nerv{n} \underset{C^\infty_{\CG\nerv{n}}}{\otimes} C^\infty_{\CH\nerv{n}}. \]
    The coface and codegeneracies are:
    \[ (F^* \phi^n_i) := (\phi^n_i) \otimes (d^n_i)^* , \qquad (F^* \sigma^n_i) := (\sigma^n_i) \otimes (s^n_i)^*, \]
    where, in the above \( d^n_i \) and \( s^n_i \), refer to the face and degeneracy maps for \( \CH \).

    Given such a pullback, at each level \( n \in \mathbb{N} \), consider the map:
    \[ (F^n)^\# \colon \CE\nerv{n} \to (F^* \CE )\nerv{n} \qquad e \mapsto e \otimes 1 \]
    This induces a morphism at the level of chain complexes:
    \[ F^\#  \colon \Ch(\CE) \to \Ch(F^*\CW)  . \]
\end{definition}
\begin{example}
    If \(E \to \CG_0 \) is a right representation of \( \CG \) then the pullback of the associated cosimplicial module \( \CE \) along a groupoid homomorphism \( F \) is just the sections of the pullback bundles:
    \[ F^* \left( E \fiber{\pi}{\bt} \CG\nerv{n} \right) \to \CH\nerv{n} \]
\end{example}
\section{Morita invariance}\label{sec:morita invariance}
In this section we will prove our main theorems which demonstrate Morita invariance in a variety of contexts for cosimplicial modules. The key to all of our proofs is the ability to construct global left retractions out of local ones so that is what we will show in the first subsection.
\subsection{Left retracts of cosimplicial modules}
\begin{definition}
    Let \( \CE \) be a cosimplicial module on \( \CG \). Let \( \CW \) be a module over a manifold \( B \) and suppose we have a comorphism \( (L,p ) \colon \CW \coto \CE \) where \( p \colon \CG\nerv{0} \to B  \) and \( L \colon \mathcal{W} \to \mathcal{E}\nerv{0} \) are augmentations of their respective structures.
    
    A \emph{left retract} of \( \CE \) to \( \CW \) along \( (L,p) \) consists of a left retract \( \{ r^n \}_{n \in \mathbb{N}}  \) along \( p \) together with a left retract \( \{ \theta^n\}_{n \in \mathbb{N}}  \) of the cosimplicial vector space \( \CE \) to \( \CW \) along \( L \) such that the combined pairs constitute module comorphisms. In other words, we have module comorphisms:
    \[ (\theta^0, r^0) \colon \CE\nerv{0} \coto \CW \]
    and for \( n \ge 1 \)
    \[  (\theta^n, r^n) \colon \CE\nerv{n} \coto \CE\nerv{n-1} \]
    which satisfy the axioms for a left retract along \( (L,p) \).
\end{definition}
Since a left retract in the category of module comorphisms constitutes a left retract in the category of vector spaces (when thinking in terms of global sections) it follows that the existence of such a retract implies that the augmentation \( \mathcal{W} \to \CE\nerv{0} \) is acyclic at the level of global sections.

Unfortunately, it is a too optimistic to hope to construct global left retractions. Even when \( L \) is acyclic, it may be the case that no left retract exists inside of the category of module comorphisms. The basic problem is that surjective submersions do not admit global sections. 

However, it turns out that if one can construct a left retraction locally (at the level of comorphisms) then it is actually possible to construct a global left retract (at the level of \emph{vector spaces}). The trick is to take advantage of partitions of unity to piece together local retractions into global retractions.

To understand how this works, let us first suppose we have a cosimplicial module \( \CE \) on \( \CG \), a module \( \CW \) on a manifold \( B \) together with an augmentation \( (L,p) \colon \CW \to \CE\nerv{0}  \). If \( \{ U_a \}_{a \in A} \) is an open cover of \( B \) then, for each \( a \in A \), one can define restricted groupoids \( \CG_a\) where the objects are
\[ (\CG_a)_0 := p^{-1}(U_a) \] 
and the arrows are 
\[ (\CG_a)_1 := \bt\inv(p\inv(U_a))  = \bs\inv(p\inv(U_a)) . \]
We can also restrict the modules to this cover:
\[ \CW_a := \CW|_{U_a} \qquad \CE_a\nerv{n} := \CG_a\nerv{n} \]
One also obtains a restricted invariant comorphism: \( (L_a, p_a) \colon \CW_a \to \CE_a\) defined in the obvious way.
\begin{lemma}\label{lemma:local.to.global.left.retract}
    Suppose we have a module \( \CW \) over a manifold \( B \) and a cosimplicial module \( \CE \) on \( \CG \). Suppose we have a comorphism \( (L,p ) \colon \CW \to \CE\nerv{0} \) where \( \bt \circ p = \bs \circ p \). Suppose further that there exists an open cover \( \{ U_a \}_{a \in A} \) of \( B \) such that for all \( a \in A \), there exists a left retract:
    \[ \{ ( \theta^n_a , r^n_a ) \}_{n \in \mathbb{N}}  \]
    of the cosimplicial module \( \CE_a \) along \( (L_a,p_a) \). Then there exists a left retract (of a cosimplicial vector space) of \( \CE \) along \( L \). Hence, the augmentation \( L \colon \CW \to \CE\nerv{0} \) is acyclic.
\end{lemma}
\begin{proof}
    Let \( \rho_a \) be a locally finite partition of unity subordinate to \( \{ U_a \}_{a \in A} \). Since \( \bt \circ p = \bs \circ p \), for each \( n \in \mathbb{N} \) there exists a partition of unity \( \{ \rho^n_a \} \) on \( \CG\nerv{n}\) subordinate to the open cover \( \{ \CG_a\nerv{n} \}_{a \in A} \). This partition of unity is compatible with face maps in the sense that:
    \[\forall a \in A ,\  1 \le i \le n ,\  e \in \CE\nerv{n-1}, \qquad  \phi^n_i ( \rho^{n-1}_a \cdot e ) = \rho^{n}_a \cdot \phi^n_i(e). \]

    In this way, we can use the partition of unity to construct a global left retract (of cosimplicial vector spaces):
\[ \theta^n := \sum_{a \in A} \rho_a\nerv{n} \theta^n_a \]
where on the right hand side we have implicitly taken advantage of the ``extension maps'' \( \CE_a \to \CE \) given by \( e \mapsto \rho_a e \).

Since the partitions of unity are all compatible with the cosimplicial structure, it follows that \( \theta^n \) is a left retract of cosimplicial vector spaces along \( L \).
\end{proof}
\subsection{Morita invariance of groupoid cohomology}\label{subsection:morita.invariance.groupoid.cohomology}
Before we prove Morita invariance in the more general context, it is useful to explain the proof the ``base case'' of groupoid cohomology. 

The following result is well known, c.f. \cite{Cra03}, however, we will make reference to the proof of this theorem multiple times later. 
\begin{theorem}\label{theorem:morita.groupoid.cohomology}
    Suppose \( F \colon \CH \to \CG \) is a weak equivalence. Then \( F^* \colon \Ch(\CG) \to \Ch(\CH) \) is a quasi-isomorphism.
\end{theorem}
\begin{proof}
    Recall that by Lemma~\ref{lemma:diagram.lemma.algebras}, we have a commutative diagram:
\[
    \begin{tikzcd}
        \Ch(\CG) \arrow[r, "L"]   & \Ch(\CG)_{{\shift}} \arrow[r, "F^*_{{\shift}}"] & \DC_F \\
        & \Ch(\CG) \arrow[r, "F^*"] \arrow[u, "B"] & \Ch(\CH) \arrow[u, "P_F^*", swap]
    \end{tikzcd}.
\]
It bears recalling the components of this diagram:
\begin{itemize}
\item The maps \( L \) and \( B \) are the standard left and bottom augmentations which map \( \Ch(\CG) \) into the shift double complex \( \Ch(\CG)_{{\shift}} \). We showed in Lemma~\ref{lemma:unit.is.retract.of.shift} that they are quasi-isomorphisms. 
\item The map \( F_{{\shift}}^* \) is the pullback along the smooth double groupoid homomorphism:
\[F_{{\shift}} \colon   \DG_F  \to \CG_{{\shift}}\]
from the \( F \)-double groupoid to the shift double groupoid of \( \CG\).
\item The map \( F^* \) is just the nerve-wise pullback map along \( F \).
\item The arrow \( P_F^* \) is the one which appeared in Lemma~\ref{lemma:diagram.lemma.algebras}. It arises from the natural projections: 
\[ P_F\nerv{m} \colon \CG\nerv{1} \fiber{\bs}{F_0 \circ \bt} \CH\nerv{m} \to \CH\nerv{m} . \]
In Lemma~\ref{lemma:diagram.lemma.algebras}, we proved that \( P_F^* \) is a quasi-isomorphism.
\end{itemize}
To prove the theorem, it suffices to show that \( F^*_{\shift} \circ L  \) is a quasi-isomorphism. This is what we will now focus on demonstrating.

Let us fix \( n \in \mathbb{N} \) and let us consider the \( n \)-th row of \( \DG_F\).
\[ 
\begin{tikzcd}
 \CG\nerv{n+1} \fiber{\bs}{F_0 } \CH\nerv{0} & \arrow[l, shift left] \arrow[l, shift right] \CG\nerv{n+1} \fiber{\bs}{F_0 \circ \bt} \CH\nerv{1} &  \arrow[l, shift left, shift left] \arrow[l, shift right, shift right] \arrow[l] \CG\nerv{n+1} \fiber{\bs}{F_0 \circ \bt}  \CH\nerv{2} & \cdots  \arrow[l, shift left] \arrow[l, shift right]  \arrow[l, shift left, shift left, shift left] \arrow[l, shift right, shift right, shift right] 
\end{tikzcd}
\]
By composing projection to the first component with the first face map \( d_0 \) we get a submersion:

\[ p \colon \CG\nerv{n+1} \fiber{\bs}{F_0 } \CH\nerv{0} \to \CG\nerv{n},  \qquad p(g_{n+1}, \ldots , g_1, x) = (g_{n+1}, \ldots , g_2).  \]
If we unpack the definition, we see that, at level \( n \), the map \( F^*_{\shift} \circ L \) is just the pullback along \( p \):

\[ F^*_{\shift} \circ L = p^* \colon C^\infty(\CG\nerv{n}) \to C^\infty( \CG\nerv{n+1} \fiber{\bs}{F_0} \CH\nerv{0} ). \]
The pullback \( p^* \) defines a left augmentation of the double complex \( \DC_F \).
By Lemma~\ref{lemma:augmentation.lemma}, if we can show this augmentation is acyclic then the proof is finished. We will show it is acyclic by constructing a left retract of each \( n \) column of \( \DC_F \) along \( p^* \).

Now, according to Lemma~\ref{lemma:local.to.global.left.retract}, we only need to construct a left retract locally since we can put them together using a partition of unity. Notice that \( \CG\nerv{n} \) can be covered by open subsets of the form \( U\nerv{n} \) where \( U \subset \CG_0 \). Hence, it suffices to construct a retract over an open subset of this form.

Let \( U \subset \CG\nerv{0} \) be an arbitrary open subset for which the map:
\[ \CG\nerv{1} \fiber{\bs}{F_0} \CH\nerv{0} \to \CG\nerv{0} ,\qquad (g,x) \mapsto \bt(g) \]
admits a section.

Such a section would take the form of a function:
\[ ( \sigma , \overline{f}) \colon U \to \CG\nerv{1} \fiber{\bs}{F_0} \CH_0, \qquad x \mapsto (\sigma(x), \overline{f}(x)),  \]
where \( \sigma \colon U \to U\nerv{1} \) is a section of the target map and \( \overline{f} \colon U \to \CH\nerv{0} \) is a smooth function which makes the following diagram commute:
\[\begin{tikzcd}
    U \arrow[d, "\overline{f}"] \arrow[r, "\sigma"] & \CG\nerv{1} \arrow[d, "\bs"] \\
    \CH\nerv{0} \arrow[r, "F_0"] & \CG\nerv{0} 
\end{tikzcd}\]

Notice that if we restrict the \( n \)th-row of \( \DG_F \) to \( p^{-1}(U\nerv{n}) \) the resulting groupoid is of the form:
\[ 
\begin{tikzcd}
U\nerv{n+1} \fiber{\bs}{F_0 } \CH\nerv{0} & \arrow[l, shift left] \arrow[l, shift right] U\nerv{n+1} \fiber{\bs}{F_0 \circ \bt} \CH\nerv{1} &  \arrow[l, shift left, shift left] \arrow[l, shift right, shift right] \arrow[l] U\nerv{n+1} \fiber{\bs}{F_0 \circ \bt}  \CH\nerv{2} & \cdots  \arrow[l, shift left] \arrow[l, shift right]  \arrow[l, shift left, shift left, shift left] \arrow[l, shift right, shift right, shift right] 
\end{tikzcd}
\]
We will now construct a left retract \( \{ \rho^m \}_{m \in \mathbb{N}} \) for this groupoid. The idea behind the construction is, essentially, to ``lift'' the target family associated to \( \sigma \) to the rows of \( \DG_F \).

Given \(m \ge 1\), let \( \rho^m \) be defined as follows:
\[ \rho^m \colon U\nerv{n+1} \fiber{\bs}{F_0 \circ \bt} \CH\nerv{m-1} \to U\nerv{n+1} \fiber{\bs}{F_0 \circ \bt} \CH\nerv{m}, \]
\[ \rho^m(g_{n+1},\ldots , g_1, h_m , \ldots , h_1) = (g_{n+1},\ldots , g_1, h_m,\ldots, h_1, \overline{h} ) \]
where \( \overline h \in \CH\nerv{1} \) is the unique element which satisfies:
\[ F(\overline{h}) = (g_{n+1} \cdots g_1 \cdot F(h_m) \cdots F(h_1))^{-1} r^0(t(g_{n+1}))) \]
and
\[ \bt(h) = \bs(h_1) \qquad \bs(\overline{h}) = \overline{f}(\bt(g_{n+1})) \]
This map is well defined since \( F \) is a weak equivalence (elements in \( \CH\nerv{1} \) are uniquely determined by their source, target, and image in \( \CG \)).
For the case \( m = 0  \) we define:
\[ \rho^0 \colon U\nerv{n} \to U\nerv{n+1} \fiber{\bs}{F_0 \circ \bt} \CH\nerv{0} \] 
\[\rho^0(g_n,\ldots, g_1) = (g_n,\ldots, g_1, (g_n \cdots g_1)^{-1}  \sigma(\bt (g_n)), \overline{f}(\bt(g_n))  )\]

 A straightforward computation shows that \( \{ \rho^m \}_{m \in \mathbb{N} } \) constitutes a left retract along \( p \). At the level of groupoid cochains we get a left retract \( \{ (\rho^m)^* \}_{m \in \mathbb{N}}\) of the cosimplicial vector space \( C^\infty(\DG\nerv{n,\bullet} )  \) which implies that the augmented row \( \DC\nerv{n,\bullet} \) is acyclic.  
\end{proof}
\begin{remark}
    Notice that, in the proof of Theorem~\ref{theorem:morita.groupoid.cohomology}, the construction of the retraction \(  \rho \) is done in a way that makes it compatible with the target family induced by \( \sigma \). To be precise, if \( \{ r^n \colon U\nerv{n} \to U\nerv{n+1} \} \) is the target family associated to \( \sigma \) then we have a commutative diagram:
    \[ 
    \begin{tikzcd}
        U\nerv{n+1} \fiber{\bs}{F_0} \CH\nerv{m} \arrow[r, "F_{\shift}"] & U\nerv{n+m+1} \\
        U\nerv{n} \fiber{\bs}{F_0} \CH\nerv{m-1} \arrow[u, "\rho^m"] \arrow[r,  "F_{\shift} "] & U\nerv{n+m} \arrow[u,"r^{n+m}"]
    \end{tikzcd}
    \]
    The existence of this lift is a consequence of the assumption that \( F \) is a weak equivalence.
\end{remark}
\subsection{Morita Invariance for sheaves on the big site}
\begin{definition}
    A \emph{big site sheaf of modules} \( \CE \) is a sheaf on the site of smooth manifolds with values in module comorphisms. If \( \CE \) is a big site sheaf of modules and \( \CG \) is a Lie groupoid, then we get a cosimplicial module \( \CE(\CG) \) by applying \( \CE \) to the entire nerve of \( \CG \).
\end{definition}
Note that any cosimplicial module arising from a sheaf on the big site will be automatically a good cosimplicial module.
\begin{example}
    If we fix \( k \in \mathbb{N} \). The sheaf of differential \(k\)-forms \( \Omega^k \) is a big site sheaf. If \( \CG \) is a Lie groupoid then \( \Omega^k(\CG) \) will be the cosimplicial vector space of \( k \)-forms on \( \CG \). If \( k = 0 \) then this is just the usual cosimplicial module of functions on \( \CG \).
\end{example}
\begin{theorem}
    Let \( \CG \) and \( \CH \) be Lie groupoids and let \( \CE \) be a big site sheaf of modules. Let \( \CE_\CG \) and \( \CE_\CH \) be the associated cosimplicial sheaves of modules. If \( F \colon \CH \to \CG \) is a weak equivalence then \( \CE_F \colon \CE_\CH \to \CE_\CG \) is a quasi-isomorphism.
\end{theorem}
\begin{proof}
The proof is essentially obtained by applying the functor \( \CE \) to the constructions in the proof of Theorem~\ref{theorem:morita.groupoid.cohomology}.

Recall that in the proof of Theorem~\ref{theorem:morita.groupoid.cohomology} we considered the diagram:
    \[
    \begin{tikzcd}
    \CG  & \arrow[l, "L", swap]    \CG_{{\shift}} \arrow[d, "B", swap]  & \arrow[l, "F_{{\shift}}", swap] \arrow[d, "P_F"]  \DG_F\\
       &   \CG   & \arrow[l, "F", swap] \CH 
    \end{tikzcd}
    \]
and observed/proved that \( L \), \( B \) and \( P_F \) all admit retractions. Furthermore, we showed that \( L \circ F_{\shift} \) admits retractions locally. 

Now we can apply the functor \( \CE \) to the above diagram to obtain a new diagram:

\begin{equation*}
    \begin{tikzcd}
        \CE_\CG \arrow[r, "{\CE_{L}}"]   & \CE_{\CG_{\shift}} \arrow[r, "\CE_{F_{{\shift}}}"]  & \CE_{\DG} \\
        & \CE_\CG \arrow[r, "\CE_{F} "] \arrow[u, "\CE_{B}"] & \CE_\CH \arrow[u, "\CE_{P_F}", swap]
    \end{tikzcd}.
\end{equation*}
Furthermore, we can apply the functor to the constructions of the left retracts for the original diagram to obtain left retracts for \( \CE_{L} \), \( \CE_{B} \) and \( \CE_{P_F} \). We also obtain local left retracts for \( \CE_{F_{\shift}}\CE_{L}\). In Lemma \ref{lemma:local.to.global.left.retract} we observed that existence of local retracts suffices to conclude that \( \CE_{F_{\shift}} \CE_L \) is acyclic.

Therefore, all of the augmentations appearing in the above diagram are acyclic. Hence, they are quasi-isomorphisms. It follows that \( \CE_{F_{\shift}} \) is a quasi-isomorphism and therefore \( \CE_{F} \) is a quasi-isomorphism.
\end{proof}
\begin{corollary}
    Let \( k \) be a positive integer and \( \CG \), \( \CH \) be Lie groupoids. If \(F \colon \CH \to \CG \) is a weak equivalence then \( F^* \colon \Omega^k(\CG) \to \Omega^k(\CH) \) is a quasi-isomorphism of cosimplicial modules.
\end{corollary}
\subsection{Morita invariance for pullbacks}
   In this section, let us fix a groupoid homomorphism \( F \colon \CH \to \CG \). Suppose we are given a cosimplicial module \( \CE \) on \( \CG \). Notice that there is a natural morphism of cosimplicial vector spaces 
\[ F^\# :=\Id \otimes 1 \colon \CE \to  \CE \otimes_\CA \CB , \qquad e \mapsto e \otimes 1  .\] 
Hence, one obtains a morphism of the associated chain complexes. Our main theorem says that this morphism is a quasi-isomorphism when \( F\) is a weak equivalence and when \( \CE \) is good.
\begin{theorem}\label{thm:weakequivalencecosimplicial}
    Let \( F \colon \CH \to \CG \) be a weak equivalence of groupoids and suppose that \( \CE \) is a good cosimplicial module on \( \CG \). Then the pullback 
    \[F^\#  \colon \CE \to F^* \CE  \] 
    defines an isomorphism in cohomology.
\end{theorem}
\begin{proof}
Recall that by Lemma~\ref{lemma:diagram.lemma.algebras}, that we have a commutative diagram:
    \[
    \begin{tikzcd}
    \CG  & \arrow[l, "L", swap]    \CG_{{\shift}} \arrow[d, "B", swap]  & \arrow[l, "F_{{\shift}}", swap] \arrow[d, "P_F"]  \DG_F\\
       &   \CG   & \arrow[l, "F", swap] \CH 
    \end{tikzcd}
    \]
    where \( L \), \( B \) and \( P_F \) are all augmentations that admit left retractions.

    At the level of groupoid cochains, we have a diagram:
    \begin{equation}\label{diagram:pullback.theorem.2}
    \begin{tikzcd}
        \Ch(\CG) \arrow[r, "L"]   & \Ch(\CG)_{{\shift}} \arrow[r, "F^*_{{\shift}}"] & \DC_F \\
        & \Ch(\CG) \arrow[r, "F^*"] \arrow[u, "B"] & \Ch(\CH) \arrow[u, "P_F^*", swap]
    \end{tikzcd}
\end{equation}
    where \( L\), \( B \) and \( P_F^* \) are quasi-isomorphisms.
By introducing a cosimplicial module \( \CE \) on \( \CG \) we can modify Diagram~\ref{diagram:pullback.theorem.2} to get a new diagram:
    \[
    \begin{tikzcd}
        \CE \arrow[r, "L_\CE"]   & \CE_{{\shift}} \arrow[r, "F_{{\shift}}^\#"] & F_\text{shift}^* \CE_{{\shift}} \\
        & \CE \arrow[r, "F^\#"] \arrow[u, "B_\CE"] &  F^* \CE \arrow[u, "B_\CE {\otimes} P_F^*", swap]
    \end{tikzcd}.
\]

\begin{itemize}
\item As before the maps \( L_\CE \) and \( B_\CE \) are the left and bottom inclusions of \( \CE \) into its associated shift double structure \( \CE_{{\shift}} \).
\item The map \( F_{{\shift}}^\# = \Id_{\CE_\text{shift}} \otimes 1  \) is the ``pullback comorphisms'' obtained from pulling back the modules \( \CE_{{\shift}} \) on the shift double groupoid of \( \CG \) to the \( F\)-double groupoid via \( F_{\shift} \). 
\item The map \( F^\#  = \Id_\CE \otimes 1 \) is the pullback morphism along \( F \colon \CH \to \CG \).
\item The vertical arrow on the right side is obtained by taking the tensor product of the bottom inclusion for \( \CE \) together with the pullback inclusion \( P_F^* \colon \CB \to \DC_F \). We leave it to the reader to verify that this is well-defined (i.e. compatible with the tensor products over \( \Ch(\CG) \) and \( \Ch(\CG_\text{shift}) \)).
\end{itemize}
We claim  that if \( \CE \) is good and \( F \) is a weak equivalence then \( F^\# \) is a quasi-isomorphism.

We have already seen that the bottom and left inclusions into the shift double complex are quasi-isomorphisms. If we can show \( F^\#_{{\shift}} \circ L_\CE \) is a quasi-isomorphism and \( B_\CE \otimes P_F^* \) is a quasi-isomorphism then we are finished. We will prove these facts in two separate lemmas.
\end{proof}
\begin{lemma}\label{lemma:BPadmitsleftretract}
The map \( B_L \otimes P_F^* \) admits a left retract and is hence a quasi-isomorphism.
\end{lemma}
\begin{proof}
Let us fix \( m \in \CH\nerv{m} \) and consider the corresponding column of \( \CE_{{\shift}} \otimes \DC_F \) as illustrated below:
\[
\begin{tikzcd}[arrows=<-]
    \vdots \fourdarrow \\
    \CE\nerv{m+3} \underset{\Ch(\CG)\nerv{m+3}}\otimes \DC_F\nerv{2,m} \threedarrow \\
    \CE\nerv{m+2} \underset{\Ch(\CG)\nerv{m+2}}\otimes \DC_F\nerv{1,m}  \twodarrow\\
    \CE\nerv{m+1} \underset{\Ch(\CG)\nerv{m+1}}\otimes \DC_F\nerv{0,m}  .  
\end{tikzcd}
\]
In Lemma~\ref{lemma:diagram.lemma.algebras}, we constructed a left retract (of groupoids):
\[  \{r^n   \}_{\mathbb{N}}  \] 
of the \( m \)-column of \( \DG_F \) to \( \CH\nerv{m} \) along \( P_F \). Notice that if we apply \( F_{\shift} \) to this left retract of groupoids we obtain the ``standard'' left retract of the \( m \)-column of \( \CG_{\shift} \) to \( \CG\nerv{m} \) along \( d^{m+1}_{m+1}\) (see Example~\ref{example:standard.left.retract}):
\[  \{ s^{n+m}_{m} \colon \CG\nerv{n+m} \to \CG\nerv{n+m+1}  \}_{n \in \mathbb{N}}  . \]

In Lemma~\ref{lemma:unit.is.retract.of.shift} we also constructed the corresponding standard left retract for cosimplicial vector spaces:
\[ \{ \sigma^{n+m}_m \colon \CE\nerv{m+n} \to \CE\nerv{m+n+1} \}_{n \in \mathbb{N}}  \] 
of the \( m \)-column of \( \CE_{\shift} \) to \( \CE\nerv{m} \) along \( \phi^{m+1}_{m+1} \) .
The standard left retract of groupoids is compatible with this left retract of modules in the sense that it is the ``base map'' for this left retract of modules. This compatibility allows us to conclude that the following is well defined (meaning that linear maps below are compatible with the relevant tensor product structure):
\[  \{ \theta^n \otimes (r^n)^* \}_{n \in \mathbb{N}} \]
Each component of the tensor product is already known to constitute a left retract. The \( \theta \) component is a left retract along \( \phi^m_m \) while the \( (r^n)^* \) component is a left retract along \( P_F \). Hence, the tensor product will also satisfy the left retract relations and will constitute a left retract along the bottom augmentation \( B_\CE \otimes P_F^*\). Since \( B_\CE \otimes P_F^* \) is injective it must be a quasi-isomorphism.
\end{proof}
\begin{lemma}\label{lem:FLadmitsleftretract}
    The map \(F^\#_{\shift} \circ L_\CE \) admits a left retract and is hence a quasi-isomorphism.
\end{lemma}
\begin{proof}
Let us fix \( n \in \mathbb{N} \) and consider the \( n \)-th column of \( \DG_F\),
\[ 
\begin{tikzcd}
 \CG\nerv{n+1} \fiber{\bs}{F_0 } \CH\nerv{0} & \arrow[l, shift left] \arrow[l, shift right] \CG\nerv{n+1} \fiber{\bs}{F_0 \circ \bt} \CH\nerv{1} &  \arrow[l, shift left, shift left] \arrow[l, shift right, shift right] \arrow[l] \CG\nerv{n+1} \fiber{\bs}{F_0 \circ \bt}  \CH\nerv{2} & \cdots  \arrow[l, shift left] \arrow[l, shift right]  \arrow[l, shift left, shift left, shift left] \arrow[l, shift right, shift right, shift right] 
\end{tikzcd}.
\]
This groupoid has a natural projection 
\[ p \colon \CG\nerv{n+1} \fiber{\bs}{F_0 } \CH\nerv{0} \to \CG\nerv{n} , \qquad p(g_{n+1}, \ldots , g_1, x) = (g_{n+1}, \ldots , g_2).  \]
At the level of modules, the \( n \)-th row of \( F_{\shift}^* \CE_\text{shift} \) is a cosimplicial module over the \( n \)-column of \( \DG_F \). It looks like:
\[ 
\begin{tikzcd}[arrows=<-]
 \CE\nerv{n+1} \underset{\Ch(\CG)\nerv{n+1}}{\otimes} \DC_F\nerv{n,0} & \arrow[l, shift left] \arrow[l, shift right] 
 \CE\nerv{n+2} \underset{\Ch(\CG)\nerv{n+2}}{\otimes} \DC_F\nerv{n,1} &  \arrow[l, shift left, shift left] \arrow[l, shift right, shift right] \arrow[l] 
 \cdots
\end{tikzcd}.
\]
The map \( L :=  F^\#_\text{shift} \circ L_\CE \) corresponds to the map:
\[ \CE\nerv{n} \to  \CE\nerv{n+1} \underset{\Ch(\CG)\nerv{n+1}}{\otimes} \DC_F\nerv{n,0}, \qquad e \mapsto \phi^{n+1}(e) \otimes 1 . \]
The pair 
\[ (L, p) \colon  \CE\nerv{n} \coto F^*_{\shift} \CE  \] 
constitutes a module comorphism. According to Lemma~\ref{lemma:local.to.global.left.retract}, if we can locally construct a left retraction along \( (L,p) \) then it follows that the cohomology of the \( n \)-th row is concentrated in the \( 0 \)-column and is equal to the image of \( L \) in degree zero. This would complete the proof. 

Hence, it suffices to complete this module comorphism to a left retract locally in \( \CG\nerv{n} \). Similar to the proof of Theorem~\ref{theorem:morita.groupoid.cohomology}, let \( U \subset \CG\nerv{0} \) be an open subset which admits a local section:
\[ ( r^0, \overline{f}) \colon U \to \CG\nerv{1} \fiber{\bs}{F_0} \CH_0 , \qquad x \mapsto (r^0(x), \overline{f}(x)).  \]

In the proof of Theorem~\ref{theorem:morita.groupoid.cohomology}, we constructed a left retract \( \{ \rho^m \}_{m \in \mathbb{N}} \) for the restriction of the \( n \)-th row to \( U \nerv{n} \). We just need to promote this left retraction to a module comorphism.

Let \( \{ r^k \colon U\nerv{k} \to U\nerv{k+1} \}_{k\in \mathbb{N}} \) be the associated target family induced by \( r^0 \). Since \( \CE \) is assumed to be good, let 
\[ \{ (r^k)^\# \colon \CE\nerv{k+1}|_{U\nerv{k+1}} \to \CE\nerv{k}|_{U\nerv{k}} \}_{k \in \mathbb{N}} \] 
be a lift of this target family to the level of the module. 

We can put \( \rho^m \) together with the lifts of the target family to obtain a family of module comorphisms parameterized by \( m \in \mathbb{N}\):
\[ \left\{ \left( r^{n+1+m})^\# \otimes (\rho^m)^* , \rho^m \right) \right\}_{m \in \mathbb{N}} \]
\end{proof}

\section{Cosimplicial Complexes over groupoids}
\label{sec:cosimplicial complexes over groupoids}
\subsection{Triple Complexes}
First, let us establish some terminology. Thus far, we have used the words ``vertical" and ``horizontal'' to refer positions/directions relative to the written page. We will now use the word ``normal'' to refer to the direction orthogonal to the page. 

We think of our triple (cochain) complexes as being bounded below by zero in each index. These triple complexes \( E^{\bullet, \bullet, \bullet}  \) are equipped with three commuting differentials: 
\[ d_V \colon E^{\bullet, \bullet, \bullet} \to E^{\bullet+1, \bullet, \bullet} \]
\[ d_H \colon E^{\bullet, \bullet, \bullet} \to E^{\bullet, \bullet+1, \bullet} \]
\[ d_N \colon E^{\bullet, \bullet, \bullet} \to E^{\bullet, \bullet, \bullet+1} \]
These differentials will be called the vertical, horizontal, and normal differentials (respectively).

The most important cohomology (and the one we are most interested in computing) is the analogue of the total cohomology.

\begin{definition}
    Suppose \( \CE \) is a triple complex. The total complex of a triple complex is defined such that the \( k\)th homogeneous component is given by:
    \[ \CE_{tot}^{k} := \sum_{a+b+c =k } \CE^{a,b,c}. \]
    The total differential is defined so that on the subspace \( \CE^{a,b,c}\) we have:
    \[ d_{tot}|_{\CE^{a,b,c}} = d_V + (-1)^a d_H + (-1)^{(a+b)} d_N. \]
\end{definition}
A routine check verifies that \( d_{tot} \) squares to zero.

Triple complexes can be thought of as double complexes \( D\nerv{n,m} \) where each \( D\nerv{n,m} \) is itself a complex rather than a vector space. However, it is not true that the total cohomology of the triple complex can be computed by first computing the total cohomology of \( D\nerv{n,m} \). For this reason we are better off treating the topic in a way that does not have any ``privileged'' directions.

Despite the previous remark, there is still a version of the augmentation lemma for triple complexes. In order to state this properly we should define augmentations in the context of triple complexes.
\begin{definition}
    Suppose \( \CE \) is a triple complex and \( D \) is a double complex. A \emph{left-normal augmentation} (LN-augmentation) 
    \[ LN \colon D \to \CE  \] consists of a morphism of double complexes:
    \[ LN \colon V \to D^{\bullet,0, \bullet} \]
    with the property that \( d_H L = 0 \).
    
    Bottom-normal and left-bottom augmentations are defined similarly (each just corresponds to changing which part of the triple complex is set to zero.)

    We say a left-normal augmentation is \emph{acyclic} if for all \( m,n \in \mathbb{N} \) we have that the following augmented complex is exact:
    \[
    \begin{tikzcd}
        D^{m,n} \arrow[r, "LN"] & E^{m,0,n} \arrow[r, "d_H"] & E^{m,1,n} \arrow[r, "d_H"] & \cdots 
    \end{tikzcd}
    \]
\end{definition}
Let us now state the augmentation lemma for triple complexes.
\begin{lemma}\label{lemma:augmentation.for.triple.complexes}
    Suppose \( E \) is a triple complex and suppose \( LN \colon D \to E \) is an acyclic left-normal augmentation of \( E \). 
    
    Then \( LN \) induces an isomorphism \( L_* \colon  H(D) \to H(E) \).
\end{lemma}

We will not include a proof of this lemma here. Instead, we will simply remark that this fact can be proved as a corollary of the two-dimensional augmentation lemma. The idea is to collapse the triple complex by ``combining'' the left-normal direction into a single complex to view \( E \) as a double complex instead.

Of course, by symmetrical arguments, one has corresponding versions of the augmentation lemma for left-bottom and bottom-normal augmentations.

One should also take note of the fact that determining whether an augmentation is acyclic really only depends on the exactness of the one dimensional complexes orthogonal to the augmentation. 

The following fact is rather clear from the definition but it is worth pointing out. Roughly, it says that, to check if an augmentation is acyclic, one only needs to verify this for each ``level'' of the normal direction.
\begin{lemma}\label{lem:tripleaugmentationacyclic}
    Suppose \( E \) is a triple complex and \( LN \colon D \to E \) is a left-normal augmentation. Then \( E \) is acyclic if and only if, for each \( n \in \mathbb{N} \) we have that the associated augmentation of double complexes:
    \[ LN^{\bullet, n} \colon D^{\bullet, n} \to E^{\bullet, \bullet , n } \]
    is acyclic. 
\end{lemma}
There is nothing particularly special about the normal direction here. Similar facts can be established for various permutations of fixed indices and types of augmentations. However, the above form will turn out to be the one most relevant to our arguments in the next section.

\subsection{Cosimplicial complexes}

\begin{definition}\label{definition:cosimplicial.complex}
    A \emph{cosimplicial complex} over a Lie groupoid \( \CG \) consists of a cosimplicial module \( \CE \) on \( \CG \) that is further equipped with the structure of a cochain complex on each level of the nerve.
    
    Therefore, for each \( n \in \mathbb{N} \) we have a (cochain) complex of the form:
    \[ \CE\nerv{n} = \bigoplus_{k \in \mathbb{N}} (\CE\nerv{n})^{k},\qquad (d_N\nerv{n})^k \colon (\CE\nerv{n})^k \to (\CE\nerv{n})^{k+1}. \]
    Furthermore, we require that this additional structure is compatible with the cosimplicial module structure. In other words, all coface maps are morphisms of complexes.

    A cosimplicial complex \( \CE \) over a Lie groupoid \( \CG \) is said to be \emph{good} if each homogeneous component \( \CE^k \) is \emph{good} as a cosimplicial module over \( \CG \)
\end{definition}
\begin{notation}
    Since complexes include an additional index, we adopt the notation \( \CE^{n,k} := (\CE\nerv{n})^{k} \) to denote the homogeneous components.

    For reasons that will be clear later, we will think of the differential \( d_N \) as corresponding to a direction ``normal'' to the page (hence the \( N \) subscript).
\end{notation}

Our model example is the differential forms of the nerve:
\begin{example}[Bott-Schulman Double Complex]
    Given a Lie groupoid \( \CG \), write \( \Omega(\CG) \) to denote the cosimplicial module of differential forms in all degrees. So:
    \[\Omega(\CG)^{n,i} := \Omega^i( \CG\nerv{n} ). \]
    The differential \( D \) is the standard de Rham differential.
\end{example}

Observe that if the compatibility of the face maps with \( d_N \) implies that \( d_N \) commutes with the simplicial differential. Hence, \( \CE \) actually has the structure of a double complex.

\begin{definition}
    Suppose \( \CE \) is a cosimplicial complex over \( \CG \). The cohomology of \( \CE \), written \( H(\CE) \) is defined to be the \emph{total} cohomology of \( \CE \) regarded as a double complex.
\end{definition}

Observe that the shifted complex \( \CE_{\shift} \) of a cosimplicial complex over \( \CG \) is no longer just a double complex but a triple complex since we have one more differential \( d_N \) which operates in the normal direction.

\subsection{Morita invariance for pullbacks}
We will now prove the analogue of Theorem~\ref{thm:weakequivalencecosimplicial} except in the context of cosimplicial complexes rather than modules.

Before proceeding with the statement, we remark that cosimplicial complexes can be pulled back along groupoid homomorphisms. The definition is the same as the definition for cosimplicial modules:

Given \( \CE \) a cosimplicial complex on \( \CG \) and a groupoid homomorphism \( F \colon \CH \to \CG \) then:
\[ (F^* \CE)^{m,n} := \CE\nerv{m,n} \otimes_{C^\infty(\CG\nerv{m})} C^\infty(\CH\nerv{m}).  \]
\begin{theorem}\label{theorem:morita.invariance.for.pullback.complexes}
    Suppose \( \CE \) is a good cosimplicial complex over \( \CG \) and \( F \colon \CH \to \CG \) is a weak equivalence of Lie groupoids. Then the natural map:
    \[ F^\# \colon \CE \to F^* \CE ,\qquad e \mapsto e \otimes 1 \]
    is an isomorphism in cohomology.
\end{theorem}
\begin{proof}
This theorem is essentially a corollary of the same result for good cosimplicial modules. To see why, let us recall how the proof for the case of a cosimplicial module went. First, we consider the commutative diagram:
    \[
    \begin{tikzcd}
    \CG  & \arrow[l, "L", swap]    \CG_{{\shift}} \arrow[d, "B", swap]  & \arrow[l, "F_{{\shift}}", swap] \arrow[d, "P_F"]  \DG_F\\
       &   \CG   & \arrow[l, "F", swap] \CH 
    \end{tikzcd}
    \]
which relates \( \CG \) to the shift double groupoid and the double groupoid associated to \( F \). Using \( \CE \), we construct a closely related diagram of complexes:
\begin{equation}
    \begin{tikzcd}
        \CE \arrow[r, "LN^\#"]   & \CE_{{\shift}} \arrow[r, "F^*_{{\shift}}"]  & F^\#_{\shift} \CE_{\shift}\\
        & \CE \arrow[r, "F^\#"] \arrow[u, "BN^\#"] & F^* \CE \arrow[u, "P_F^\#", swap]
    \end{tikzcd}
\end{equation}
Now, in the setting where \( \CE \) is a cosimplicial complex, we see that \( \CE_{\shift} \) and \( F^*_{\shift} \CE_{\shift} \) are now \emph{triple complexes}. Furthermore, \( LN^\# \) is a left-normal augmentation. \( BN^\#\) and \( P_F^\#\) are bottom-normal augmentations.

If we restrict these complexes to one level of the ``normal'' direction then this is the same diagram that appears in the proof of the cosimplicial modules case (Theorem~\ref{thm:weakequivalencecosimplicial}). Furthermore, in that same proof we established that all of the augmentations in this diagram are acyclic. In Lemma ~\ref{lem:tripleaugmentationacyclic} we observed that the property of being acyclic can be verified by restricting to one level of the normal direction. Therefore, it follows that the following augmentations are acyclic:
\[ LN^\# , BN^\#, P_F^\# , F^\#_{\shift} \circ LN^\# \]
In particular, they are all quasi-isomorphisms. From this it follows that \( F^\#_{\shift} \) is a quasi-isomorphism and hence \( F^\# \) is a quasi-isomorphism.
\end{proof}
\subsection{Morita Invariance for the big site}

We will now prove the analogue of Theorem~\ref{thm:weakequivalencecosimplicial} in the context of cosimplicial complexes.

To state the theorem properly, we require the notion of a sheaf of complexes on the big site.

\begin{definition}\label{defn:sheaf of complexes on the big site}
    A \emph{sheaf of complexes on the big site} consists of a sheaf of modules \( \CE \) on the big site that is further enriched with a compatible sheaf of complexes structure.

    Given such a sheaf, every Lie groupoid \( \CG \) is equipped with a cosimplicial complex \( \CE_\CG \) and a corresponding cohomology \( H(\CE_\CG) \).

    If \( F \colon \CH \to \CG \) is a homomorphism of Lie groupoids we write \( \CE_F \colon \CE_\CG \to \CE_\CH \) to denote the associated morphism of cosimplicial complexes.
\end{definition}

\begin{theorem}\label{thm:big.site.for.complexes}
    Suppose \( \CE \) is a sheaf of complexes on the big site. \( F\colon \CH \to \CG \) is a weak equivalence of Lie groupoids, then the map:
    \[ \CE_F \colon \CE_\CH \to \CE_\CG \]
    induces an isomorphism in cohomology.
\end{theorem}
\begin{proof}
This theorem is essentially a corollary of the proof for cosimplicial modules (Theorem~\ref{thm:weakequivalencecosimplicial}). To see why, let us recall how the proof for the case of a cosimplicial module went. First, we consider the commutative diagram:
    \[
    \begin{tikzcd}
    \CG  & \arrow[l, "L", swap]    \CG_{{\shift}} \arrow[d, "B", swap]  & \arrow[l, "F_{{\shift}}", swap] \arrow[d, "P_F"]  \DG_F\\
       &   \CG   & \arrow[l, "F", swap] \CH 
    \end{tikzcd}
    \]
which relates \( \CG \) to the shift double groupoid and the double groupoid associated to \( F \). We can apply our functor \( \CE \) to this entire diagram to obtain a new diagram of complexes:
\begin{equation*}
    \begin{tikzcd}
        \CE_\CG \arrow[r, "{\CE_{LN}}"]   & \CE_{\CG_{\shift}} \arrow[r, "\CE_{F_{{\shift}}}"]  & \CE_{\DG} \\
        & \CE_\CG \arrow[r, "\CE_{F} "] \arrow[u, "\CE_{BN}"] & \CE_\CH \arrow[u, "\CE_{P_F}", swap]
    \end{tikzcd}.
\end{equation*}
Observe that in this case we are now dealing with triple complexes and augmentations of triple complexes. Now if we restrict the above diagram to a single level of the normal direction (say \( n \in \mathbb{N} \)), it becomes precisely the diagram appearing in the proof of Theorem~\ref{thm:weakequivalencecosimplicial} except applied to the functor \( \CE^n \) which only considers a single homogeneous degree of our sheaf of complexes.

In the same proof, we showed that all of the augmentations appearing in the 2D version of the above diagram are acyclic. We observed in Lemma~\ref{lem:tripleaugmentationacyclic} that it is sufficient to check that an augmentation is acyclic for individual levels of the normal direction. Hence, it follows that all augmentations (of triple complexes) that appear in the above diagram are quasi-isomorphisms. From this, it follows that \( \CE_{F_{\shift}} \) is a quasi-isomorphism and so \( \CE_{F} \) is a quasi-isomorphism.
\end{proof}

\bibliographystyle{alpha}
\bibliography{references.bib}

\end{document}